\documentclass{article}%
\usepackage{eurosym}
\usepackage{graphicx}
\usepackage{amsmath,amssymb,amsfonts}
\usepackage{theorem}
\usepackage{color}
\usepackage{hyperref}
\usepackage[font={small, it}]{caption}%
\usepackage{amsmath}%
\usepackage{amsfonts}%
\usepackage{amssymb}
\providecommand{\U}[1]{\protect \rule{.1in}{.1in}}
\theoremstyle{change}
\newtheorem{definition}{Definition:}[section]
\newtheorem{proposition}[definition]{Proposition:}
\newtheorem{theorem}[definition]{Theorem:}
\newtheorem{lemma}[definition]{Lemma:}

{\theorembodyfont{\rmfamily}
\newtheorem{remark}[definition]{Remark:}
}
{\theorembodyfont{\rmfamily}

}
\newenvironment{proof}
{{\bf Proof:}}
{\qquad \hspace*{\fill} $\Box$}

\begin{document}

\title{On the characterization of the controllability property for linear control
systems on nonnilpotent, solvable three-dimensional Lie groups}
\author{V\'{\i}ctor Ayala\\Instituto de Alta Investigaci\'{o}n\\Universidad de Tarapac\'{a}\\Sede Iquique, Iquique, Chile
\and Adriano Da Silva\\Instituto de Matem\'{a}tica\\Universidade Estadual de Campinas\\Cx. Postal 6065, 13.081-970 Campinas-SP, Brasil.}
\date{\today }
\maketitle

\begin{abstract}
In this paper we show that a complete characterization of the controllability
property for linear control system on three-dimensional solvable nonnilpotent
Lie groups is possible by the LARC and the knowledge of the eigenvalues of the
derivation associated with the drift of the system.

\end{abstract}



\section{Introduction}

Linear control systems on Euclidean spaces appear in several physical
applications (see for instance \cite{Leitmann, P-B-G-M, Shell}). A natural
extension of a linear control system on Lie groups appears first in
\cite{Markus} for matrix groups and then in \cite{AyTi} for any Lie group. In
the subsequent years, several works addressing the main problems in control
theory for such systems, such as controllability, observability and
optimization appeared (see \cite{DSAy, DSAyGZ, San, DS, DaJo, Jouan1,
Jouan2}). In \cite{Jouan2} P. Jouan shows that such generalization is also
important for the classification of general affine control systems on abstract
connected manifolds. He shows that any affine control system on a connected
manifold that generates a finite dimensional Lie algebra is diffeomorphic to a
linear control system on a Lie group or on a homogeneous space.

Concerning controllability of linear control system, in \cite{DSAy} and
\cite{DS} a more geometric approach was proposed by considering the
eigenvalues of a derivation associated with the drift of the system. In
particular, it was shown that a linear system is controllable if its reachable
set from the identity is open and the associated derivation has only
eigenvalues with zero real part. For restricted linear control systems on
nilpotent Lie groups such condition is also necessary for controllability. In
the same direction, Dath and Jouan show in \cite{DaJo} that linear control
systems (restricted or not) on a two-dimensional solvable Lie group present
the same behavior, they are controllable if and only if they satisfy the Lie
algebra rank condition and the associated derivation has only zero
eigenvalues. Here, the LARC is equivalent to the ad-rank condition which
implies, in particular, the openness of the reachable set (see \cite{AyTi}).

In the present paper, we show that the behavior of nonrestricted linear
control systems on three-dimensional solvable nonnilpotent Lie groups differ
significantly from the two-dimensional case. By using a beautiful
classification of three-dimensional solvable Lie groups (see Chapter 7 of
\cite{Oni}) we show that the geometry of the group strongly interferes in the
controllability of the system and, although such systems do not behave the
same as in the two-dimensional case, a complete characterization of their
controllability is possible only by the knowledge of the eigenvalues of the
associated derivation and the Lie algebra rank condition.

The paper is structured as follows: Section 2 is used to introduce the main
properties and results concerning linear vector fields, linear control systems
and decompositions of Lie algebras and Lie group induced by derivations.
Section 3 is devoted to the study of nonnilpotent, solvable three-dimensional
Lie groups. By using the classification in \cite{Oni}, we algebraically
characterize the main elements needed in the proofs concerning controllability
such as derivations, linear and invariant vector fields and so on. At the end
of the section, we present some particular homogeneous spaces which will be of
great importance when considering projections of linear control systems. In
Section 4 we completely characterize the controllability of linear control
systems on such groups. The work is divided into two cases, depending if the
dimension of the Lie subalgebra generated by the control vectors is one or
two, and then analyzed group by group using the classification presented in
Section 3.

\section{Preliminaries}

\subsection{Notations}

In the whole paper, the Lie groups and subgroups considered are assumed to be
connected unless we say the contrary. Their Lie algebras are identified with
the set of left-invariant vector fields. If $M,N$ are smooth manifolds and
$f:M\rightarrow N$ is a differentiable map, we denote by $(df)_{x}$ the
differential of $f$ at the point $x\in M$ and by $f_{\ast}$ the differential
of $f$ at any given point.

For any element $g\in G$ we denote by $L_{g}$ and $R_{g}$ the left and right
translations of $G$ and $e\in G$ stands for the identity element of $G$. If a
Lie algebra is given by the semi-direct product $\mathfrak{g}=\mathfrak{h}%
\times_{\theta}\mathfrak{k}$ we will use the identification $\mathfrak{k}%
=\{0\} \times \mathfrak{k}$ and $\mathfrak{h}=\mathfrak{h}\times \{0\}$ and the
same holds for Lie groups that are given as semi-direct product.

\subsection{Linear vector fields and decompositions}

In this section, we define linear vector fields and state their main
properties. For the proof of the assertions in this section the reader can
consult \cite{AyTi}, \cite{Jouan1} and \cite{Jouan2}.

Let $G$ be a connected Lie group with Lie algebra $\mathfrak{g}$. A vector
field $\mathcal{X}$ on $G$ is said to be \emph{linear} if its flow
$(\varphi_{t})_{t\in \mathbb{R}}$ is a $1$-parameter subgroup of $\mathrm{Aut}%
(G)$. Associate to any linear vector field $\mathcal{X}$ there is a derivation
$\mathcal{D}$ of $\mathfrak{g}$ defined by the formula
\begin{equation}
\mathcal{D}Y=-[\mathcal{X},Y](e),\mbox{ for all }Y\in \mathfrak{g}%
.\label{derivation}%
\end{equation}
The relation between $\varphi_{t}$ and $\mathcal{D}$ is given by the formula
\begin{equation}
(d\varphi_{t})_{e}=\mathrm{e}^{t\mathcal{D}}\; \; \; \mbox{ for all }\; \;
\;t\in \mathbb{R}.\label{derivativeonorigin}%
\end{equation}
In particular, it holds that
\[
\varphi_{t}(\exp Y)=\exp(\mathrm{e}^{t\mathcal{D}}Y),\mbox{ for all }t\in
\mathbb{R},Y\in \mathfrak{g}.
\]

The above equation implies that if $\mathcal{D}\equiv0$ we necessarily have
$\mathcal{X}\equiv0$. Since we are interested in linear systems with
nontrivial drift, we will always assume $\mathcal{D}\neq0$.

Let $G$ be a Lie group and $\widetilde{G}$ its simply connected covering. Let
$\mathcal{X}$ be a linear vector field on $G$ and $\mathcal{D}$ its associated
derivation. By Theorem 2.2 of \cite{AyTi}, there exists a unique linear vector
field $\widetilde{\mathcal{X}}$ on $\widetilde{G}$ whose associated derivation
is $\mathcal{D}$. If we denote, respectively by, $\{ \widetilde{\varphi_{t}%
}\}_{t\in \mathbb{R}}$ and $\{ \varphi_{t}\}_{t\in \mathbb{R}} $ the flows of
$\widetilde{\mathcal{X}}$ and $\mathcal{X}$ we have
\[
\pi(\widetilde{\varphi_{t}}(\exp_{\widetilde{G}} X))=\pi(\exp_{\widetilde{G}%
}(\mathrm{e}^{t\mathcal{D}}X))=\exp_{G}(\mathrm{e}^{t\mathcal{D}}%
X)=\varphi_{t}(\exp_{G}X), \; \; \mbox{ for any }\; \;X\in \mathfrak{g}%
\]
where $\pi:\widetilde{G}\rightarrow G$ is the canonical projection. By
connectedness it holds that
\[
\pi \circ \widetilde{\varphi_{t}}=\varphi_{t}\circ \pi
\; \; \mbox{ for any }\; \;t\in \mathbb{R}%
\]
implying that $\widetilde{\mathcal{X}}$ and $\mathcal{X}$ are $\pi$-related.

Next, we explicitly some decompositions of the Lie algebra $\mathfrak{g}$
induced by any given derivation $\mathcal{D}$. To do that, let us consider,
for any eigenvalue $\alpha$ of $\mathcal{D}$, the real generalized eigenspaces
of $\mathcal{D}$
\[
\mathfrak{g}_{\alpha}=\{X\in \mathfrak{g}:(\mathcal{D}-\alpha I)^{n}%
X=0\; \; \mbox{for some }n\geq1\},\; \; \mbox{ if }\; \; \alpha \in \mathbb{R}%
\; \; \mbox{
and}
\]%
\[
\mathfrak{g}_{\alpha}=\mathrm{span}\{ \mathrm{Re}(v),\mathrm{Im}(v);\; \;v\in
\bar{\mathfrak{g}}_{\alpha}\},\; \; \mbox{ if }\; \; \alpha \in \mathbb{C}%
\]
where $\bar{\mathfrak{g}}=\mathfrak{g}+i\mathfrak{g}$ is the complexification
of $\mathfrak{g}$ and $\bar{\mathfrak{g}}_{\alpha}$ the generalized eigenspace
of $\bar{\mathcal{D}}=\mathcal{D}+i\mathcal{D}$, the extension of
$\mathcal{D}$ to $\bar{\mathfrak{g}}$. By Proposition 3.1 of \cite{SM1} it
holds that $[\bar{\mathfrak{g}}_{\alpha},\bar{\mathfrak{g}}_{\beta}%
]\subset \bar{\mathfrak{g}}_{\alpha+\beta}$ when $\alpha+\beta$ is an
eigenvalue of $\mathcal{D}$ and zero otherwise. By considering in
$\mathfrak{g}$ the subspaces $\mathfrak{g}_{\lambda}:=\bigoplus_{\alpha
;\mathrm{Re}(\alpha)=\lambda}\mathfrak{g}_{\alpha}$, where $\mathfrak{g}%
_{\lambda}=\{0\}$ if $\lambda \in \mathbb{R}$ is not the real part of any
eigenvalue of $\mathcal{D}$, we get
\[
\lbrack \mathfrak{g}_{\lambda_{1}},\mathfrak{g}_{\lambda_{1}}]\subset
\mathfrak{g}_{\lambda_{1}+\lambda_{2}}\; \; \; \mbox{ when }\lambda_{1}%
+\lambda_{2}=\mathrm{Re}(\alpha)\; \mbox{ for some eigenvalue }\alpha
\; \mbox{ of }\; \mathcal{D}\; \mbox{ and zero otherwise}.
\]

This fact allow us to decompose $\mathfrak{g}$ as
\[
\mathfrak{g}=\mathfrak{g}^{+}\oplus \mathfrak{g}^{0}\oplus \mathfrak{g}^{-}%
\]
where
\[
\mathfrak{g}^{+}=\bigoplus_{\alpha:\, \mathrm{Re}(\alpha)> 0}\mathfrak{g}%
_{\alpha},\hspace{1cm}\mathfrak{g}^{0}=\bigoplus_{\alpha:\, \mathrm{Re}%
(\alpha)=0}\mathfrak{g}_{\alpha}\hspace{1cm}\mbox{ and }\hspace{1cm}%
\mathfrak{g}^{-}=\bigoplus_{\alpha:\, \mathrm{Re} (\alpha)<0}\mathfrak{g}%
_{\alpha}.
\]
It is easy to see that $\mathfrak{g}^{+},\mathfrak{g}^{0},\mathfrak{g}^{-}$
are $\mathcal{D}$-invariant Lie algebras and $\mathfrak{g}^{+}$,
$\mathfrak{g}^{-}$ are nilpotent.

At the Lie group level we will denote by $G^{+}$, $G^{-}$, $G^{0}$, $G^{+,0},
$ and $G^{-,0}$ the connected Lie subgroups of $G$ with Lie algebras
$\mathfrak{g}^{+}$, $\mathfrak{g}^{-}$, $\mathfrak{g}^{0}$, $\mathfrak{g}%
^{+,0}:=\mathfrak{g}^{+}\oplus \mathfrak{g}^{0}$ and $\mathfrak{g}%
^{-,0}:=\mathfrak{g}^{-}\oplus \mathfrak{g}^{0}$ respectively. The above
subgroups play a fundamental role in the understand of the dynamics of linear
control system as showed in \cite{DSAy}, \cite{DsAyGZ} and \cite{DS}. By
Proposition 2.9 of \cite{DS}, all the above subgroups are $\varphi$-invariant
and closed. Moreover, if $G$ is a solvable Lie group then $G=G^{+, 0}%
G^{-}=G^{-, 0}G^{+}$.

The next lemma shows that, for solvable Lie groups, the nilradical contains
all the generalized eigenspaces associated with nonzero eigenvalues.

\begin{lemma}
\label{nilradical} Let $\mathfrak{g}$ be a solvable Lie algebra and
$\mathfrak{n}$ its nilradical. If $\mathcal{D}$ is a derivation of
$\mathfrak{g}$ then, for any nonzero eigenvalue $\alpha$ of $\mathcal{D}$, it
holds that $\mathfrak{g}_{\alpha}\subset \mathfrak{n}.$
\end{lemma}

\begin{proof}
If $\alpha \in \mathbb{R}^{\ast}$ then $\mathfrak{g}_{\alpha}=\bigcup
_{n\in \mathbb{N}}\ker(\mathcal{D}-\alpha I)^{n}$. If $n=1$ we have that
$X\in \ker(\mathcal{D}-\alpha I)$ is such that $\mathcal{D}X=\alpha X.$ Since
$\mathcal{D}\mathfrak{g}\subset \mathfrak{n}$ and $\alpha \neq0$ we get
$X\in \mathfrak{n}$ implying that $\ker(\mathcal{D}-\alpha I)\subset
\mathfrak{n}$. Inductively, if $\ker(\mathcal{D}-\alpha I)^{n}\subset
\mathfrak{n}$ then, for any $X\in \ker(\mathcal{D}-\alpha I)^{n+1}$ it holds
that
\[
0=(\mathcal{D}-\alpha I)^{n+1}X=(\mathcal{D}-\alpha I)^{n}(\mathcal{D}-\alpha
I)X\implies(\mathcal{D}-\alpha I)X\in \ker(\mathcal{D}-\alpha I)^{n}%
\subset \mathfrak{n}.
\]
Using again that $\alpha \neq0$ and $\mathcal{D}\mathfrak{g}\subset
\mathfrak{n}$ gives us $(\mathcal{D}-\alpha I)X\in \mathfrak{n}$ which implies
that $X\in \mathfrak{n}$. Therefore, $\mathfrak{g}_{\alpha}\subset \mathfrak{n}$
as stated.

If $\alpha \in \mathbb{C}^{\ast}$ we have as in the real case that
$\bar{\mathfrak{g}}_{\alpha}\subset \bar{\mathfrak{n}}$, where $\bar
{\mathfrak{n}} $ is the nilradical of $\bar{\mathfrak{g}}$. Since the
conjugation in $\bar{\mathfrak{g}}$ is an automorphism we have that
$\bar{\mathfrak{n}}$ is invariant by conjugation and hence $\bar{\mathfrak{n}%
}=\mathfrak{n}^{\ast}+i\mathfrak{n}^{\ast}$ for some subspace $\mathfrak{n}%
^{\ast}\subset \mathfrak{g}$. A simple calculation shows that $\mathfrak{n}%
^{\ast}$ is in fact, a nilpotent ideal of $\mathfrak{g}$ and consequently
$\mathfrak{n}^{\ast}\subset \mathfrak{n}$. Since $\mathrm{Re}(v),\mathrm{Im}%
(v)\in \mathfrak{n}^{\ast}$ for any $v\in \bar{\mathfrak{g}}_{\alpha}$ we get
that $\mathfrak{g}_{\alpha}\subset \mathfrak{n}$ which concludes the proof.
\end{proof}

By Lemma 2.3 of \cite{DS}, the above subalgebras and subgroups are preserved
by homomorphisms in the following sense: If $\psi:G_{1}\rightarrow G_{2}$ is a
surjective homomorphism between Lie groups such that $(d\psi)_{e}%
\circ \mathcal{D}_{1}=\mathcal{D}_{2}\circ(d\psi)_{e}$, where $\mathcal{D}_{i}$
is a derivation in the Lie algebra $\mathfrak{g}_{i}$ of $G_{i}$, $i=1, 2$
then
\begin{equation}
\label{homo}(d\psi)_{e}\mathfrak{g}_{1}^{*}=\mathfrak{g}_{2}^{*}%
\; \; \mbox{ and }\; \; \psi(G_{1}^{*})=G_{2}^{*}, \; \; \mbox{ where }\; \;*=+, 0,
-.
\end{equation}

\subsection{Linear control systems}

Let $G$ be a connected Lie group and $\mathfrak{g}$ its Lie algebra identified
with the vector space of all left-invariant vector fields. A \emph{linear
control system} on $G$ is given a family of ordinary differential equations
\begin{equation}
\dot{g}=\mathcal{X}(g)+\sum_{j=1}^{m}u_{i}Y^{j}(g),\label{linearsystem}%
\end{equation}
where the \emph{drift} $\mathcal{X}$ is a linear vector field and the
\emph{control vectors} $Y^{1},\ldots,Y^{m}$ are, left-invariant vector fields.
The \emph{control functions} $u=(u_{1},\ldots,u_{m})$ belongs to
$\mathcal{U}\subset L_{\mathrm{loc}}^{\infty}(\mathbb{R},\mathbb{R}^{m})$, a
subset that contains the piecewise constant functions and is stable by
concatenations, that is, if $u_{1},u_{2}\in \mathcal{U}$ then the function $u $
defined by
\[
u(t)=\left \{
\begin{array}
[c]{cc}%
u_{1}(t), & t\in(-\infty,T)\\
u_{2}(t-T) & t\in \lbrack T,\infty)
\end{array}
\right.
\]
belongs to $\mathcal{U}$.

If $\phi(t, g, u)$ denotes the solution of (\ref{linearsystem}) associated
with $u\in \mathcal{U}$ and starting at $g\in G$ then
\[
\phi(t, g, u)=\varphi_{t}(g)\phi(t, e, u)=R_{\phi(t, e, u)}\left( \varphi
_{t}(g)\right) .
\]

The \emph{reachable set from $g$ at time} $t>0$ and the \emph{reachable set
from} $g$ are given, respectively, by
\[
\mathcal{A}_{t}(g):=\{ \phi(t,g,u),\; \;u\in \mathcal{U}%
\} \; \; \; \mbox{ and }\; \; \; \mathcal{A}(g):=\bigcup_{t>0}\mathcal{A}_{t}(g)
\]
Analogously, the \emph{controllable set to $g$ at time} $t>0$ and the
\emph{controllable set to $g$} are given, respectively, by
\[
\mathcal{A}_{t}^{\ast}(g):=\{h\in G;\; \exists u\in \mathcal{U};\; \; \phi
(t,h,u)=g\} \; \; \; \mbox{ and }\; \; \; \mathcal{A}^{\ast}(g):=\bigcup
_{t>0}\mathcal{A}_{t}^{\ast}(g)
\]
For the particular case where $g=e$ is the identity element of $G$\ we denote
the above sets only by $\mathcal{A}_{t},\mathcal{A},\mathcal{A}_{t}^{\ast}$
and $\mathcal{A}^{\ast}$, respectively.

We will say that the linear control system (\ref{linearsystem}) is
\emph{controllable} if for any $g, h\in G$ it holds that $h\in \mathcal{A}(g)$.
It is not hard to see that the controllability of (\ref{linearsystem}) is
equivalent to the equality $G=\mathcal{A}\cap \mathcal{A}^{*}$.

Let us denote by $\Delta$ the Lie subalgebra of $\mathfrak{g}$ generated by
$Y^{1}, \ldots, Y^{m}$ and by $G_{\Delta}$ its associated connected Lie
subgroup. The linear control system (\ref{linearsystem}) is said to satisfy
the \emph{ad-rank condition} if $\mathfrak{g}$ is the smallest $\mathcal{D}%
$-invariant subspace containing $\Delta$. It is said to satisfy the \emph{Lie
algebra rank condition} (LARC) if $\mathfrak{g}$ is the smallest $\mathcal{D}%
$-invariant subalgebra containing $\Delta$.

Since we are interested in the controllability of linear control systems and
the control functions are taking values in the whole $\mathbb{R}^{m}$, Theorem
3.5 in \cite{Jurd} implies that $G_{\Delta}\subset \operatorname{cl}%
(\mathcal{A})\cap \operatorname{cl}(\mathcal{A}^{\ast})$ and also that the
closures $\operatorname{cl}(\mathcal{A})$ and $\operatorname{cl}%
(\mathcal{A}^{\ast})$ remain the same if we change $Y_{1},\ldots,Y^{m}$ for
any basis of $\Delta$. Furthermore, under the LARC it holds that
$G=\operatorname{cl}(\mathcal{A}^{(\ast)})$ iff $G=\mathcal{A}^{(\ast)}$ and
therefore, if $\dim \Delta=\dim \mathfrak{g}$ the system is trivially
controllable. Since $G$ is an analytic manifold and the linear and invariant
vector fields are complete, Theorem 3.1 of \cite{Jurd2} implies that the LARC
is a necessary condition for controllability.

By the previous discussion, under the LARC the controllability of
(\ref{linearsystem}) only depends on $\mathcal{X}$ and on $\Delta$. Therefore,
we will use $\Sigma(\mathcal{X},\Delta)$ to denote the linear system with
drift $\mathcal{X}$ and control vectors given by any basis of $\Delta$, where
$\Delta$ is a proper, nontrivial subalgebra of $\mathfrak{g}$.

The next results relate the subgroups associated with the derivation induced
by $\mathcal{X}$ with the reachable and controllable sets.

\begin{lemma}
\label{translation} It holds:

1. Let $C\in \{ \mathcal{A}, \mathcal{A}^{*}, \operatorname{cl}(\mathcal{A}),
\operatorname{cl}(\mathcal{A}^{*})\}$ and $g\in G$. If $\{ \varphi_{t}(g),
t\in \mathbb{R}\} \subset C $ then $L_{g}(C)\subset C$;

2. If $\varphi_{t}(g)=g$ for all $t\in \mathbb{R}$ then $g\in \mathcal{A}$ if
and only if $g^{-1}\in \mathcal{A}^{*}$;
\end{lemma}

\begin{proof}
Item 1. is an slight modification of Lemma 3.1 of \cite{DS} and hence we will
omit its proof. For item 2., if $g\in \mathcal{A}$ there exists $t>0,
u\in \mathcal{U}$ with $g=\phi(t, e, u)$. Hence,
\[
\phi(t, g^{-1}, u)=\varphi_{t}(g^{-1})\phi(t, e, u)=\varphi_{t}(g)^{-1}%
g=g^{-1}g=e
\]
implying that $g^{-1}\in \mathcal{A}^{*}$. Reciprocally, if $g^{-1}%
\in \mathcal{A}^{*}$ then $e=\phi(t, g^{-1}, u)$ for some $t>0$, $u\in
\mathcal{U}$ and analogously
\[
e=\phi(t, g^{-1}, u)=\varphi_{t}(g^{-1})\phi(t, e, u)=g^{-1}\phi(t, e,
u)\implies g=\phi(t, e, u)\in \mathcal{A}%
\]
concluding the proof.
\end{proof}

Concerning the controllability of linear control systems we have the following
results from \cite{DS} (see Theorem 3.7)

\begin{theorem}
\label{meuteorema} If $\Sigma(\mathcal{X}, \Delta)$ is a linear system on a
solvable Lie group $G$ and assume that $\mathcal{A}$ is open, then
\[
G^{+, 0}\subset \mathcal{A}\; \; \; \mbox{ and }\; \; \;G^{-, 0}\subset
\mathcal{A}^{*}.
\]
In particular, if $\mathcal{D}$ has only eigenvalues with zero real part and
$\mathcal{A}$ is open, then $\Sigma(\mathcal{X}, \Delta)$ is controllable.
\end{theorem}

\begin{remark}
\label{importantremark} We should notice that the condition on the openness of
$\mathcal{A}$ is guaranteed, for instance, if $\Sigma(\mathcal{X}, \Delta) $
satisfies the ad-rank condition (see Theorem 3.5 of \cite{AyTi}). In
particular, if $\Delta$ has codimension one in $\mathfrak{g}$ then the LARC is
equivalent to the ad-rank condition.
\end{remark}

\begin{remark}
An extension of Theorem \ref{meuteorema} for a much larger class of Lie groups
was proved in \cite{DSAy} under the assumption that $\mathcal{A}_{\tau}$ is a
neighborhood of the identity element for some $\tau>0$.
\end{remark}

Let $G$ and $H$ be Lie groups and $\psi:G\rightarrow H$ a surjective Lie group
homomorphism. If $\mathcal{X}$ is a linear vector field on $G$ then
$\mathcal{X}_{\psi}:=\psi_{\ast}\mathcal{X}$ is a linear vector field on $H $.
Therefore, if $\Sigma(\mathcal{X},\Delta)$ is a linear control system on $G$,
by considering $\Delta_{\psi}:=(d\psi)_{e}\Delta$ we have that $\Sigma
=\Sigma(\mathcal{X}_{\psi},\Delta_{\psi})$ is a linear control system on $H$
that is $\psi$ conjugated to $\Sigma(\psi,\Delta)$. As a particular case, if
$\widetilde{G}$ is the simply connected covering of $G$ and $\Sigma
(\mathcal{X},\Delta)$ is a linear control system on $G$, the control system
$\Sigma(\widetilde{\mathcal{X}},\Delta)$ is $\pi$-conjugated to $\Sigma
(\mathcal{X},\Delta)$. The next proposition states the main relationships
between conjugated control systems.

\begin{proposition}
\label{simplesmente} Let $\Sigma(\mathcal{X}, \Delta)$ be a linear control
system on $G$ and $\psi:G\rightarrow H$ a surjective homomorphism. It holds:

\begin{itemize}
\item[1.] If $\Sigma(\mathcal{X}, \Delta)$ is controllable, then
$\Sigma(\mathcal{X}_{\psi}, \Delta_{\psi})$ is controllable;

\item[2.] If $\ker \psi \subset \operatorname{cl}(\mathcal{A})\cap
\operatorname{cl}(\mathcal{A}^{*})$ and $\Sigma(\mathcal{X}_{\psi},
\Delta_{\psi})$ is controllable, then $\Sigma(\mathcal{X}, \Delta)$ is controllable;

\item[3.] If $\Sigma(\mathcal{X}, \Delta)$ satisfies the ad-rank condition
then $\Sigma(\mathcal{X}, \Delta)$ is controllable if and only if
$\Sigma(\widetilde{\mathcal{X}}, \Delta)$ is controllable.
\end{itemize}
\end{proposition}

\begin{proof}
1. It follows directly from the fact that $\psi(\mathcal{A})=\mathcal{A}%
_{\psi}$ and $\psi(\mathcal{A}^{\ast})=\mathcal{A}_{\psi}^{\ast}$;

2. It holds that $\psi^{-1}(\mathcal{A}_{\psi})=\ker \psi \cdot \mathcal{A} $ and
$\psi^{-1}(\mathcal{A}_{\psi}^{\ast})=\ker \psi \cdot \mathcal{A}^{\ast}$. We
know that $\ker \psi$ is invariant by the flow $\{ \varphi_{t}\}_{t\in
\mathbb{R}}$ of $\mathcal{X}$, so, if $\ker \psi \subset \operatorname{cl}%
(\mathcal{A})\cap \operatorname{cl}(\mathcal{A}^{\ast})$ then $\{ \varphi
_{t}(g),t\in \mathbb{R}\} \subset \mathcal{A}\cap \mathcal{A}^{\ast}$ for all
$g\in \ker \psi$. By Lemma \ref{translation} we get that $\ker \psi
\cdot \mathcal{A}\subset \mathcal{A}$ and $\ker \psi \cdot \mathcal{A}^{\ast
}\subset \mathcal{A}^{\ast}$ and therefore, if $\Sigma(\mathcal{X}_{\psi
},\Delta_{\psi})$ is controllable we have
\[
G=\psi^{-1}(H)=\psi^{-1}(\mathcal{A}_{\psi}\cap \mathcal{A}_{\psi}^{\ast
})=(\ker \psi \cdot \mathcal{A})\cap(\ker \psi \cdot \mathcal{A}^{\ast}%
)\subset \mathcal{A}\cap \mathcal{A}^{\ast}%
\]
implying that $\Sigma(\mathcal{X},\Delta)$ is controllable.

3. If $\Sigma(\widetilde{\mathcal{X}}, \Delta)$ is controllable, then by item
1. $\Sigma(\mathcal{X}, \Delta)$ is controllable. Reciprocally, since $\ker
\pi$ is invariant by the flow of $\widetilde{\mathcal{X}}$ and it is a
discrete subgroup we must have that $\widetilde{\mathcal{X}}(\ker \pi)=0$
implying that $\ker \pi \subset \widetilde{G}^{0}$. If $\Sigma(\mathcal{X},
\Delta)$ is controllable and the ad-rank condition is satisfied, then
necessarily $\widetilde{\mathcal{A}}$ is open which by Theorem
\ref{meuteorema} implies
\[
\ker \pi \subset \widetilde{G}^{0}\subset \widetilde{\mathcal{A}}\cap
\widetilde{\mathcal{A}}^{*}%
\]
and by item 2. we have the controllability of $\Sigma(\widetilde{\mathcal{X}},
\Delta)$.
\end{proof}

We end this section with a result of Dath and Jouan characterizing the
controllability of linear control systems on the two-dimensional solvable Lie
group (see Theorem 3 of \cite{DaJo}), that it will be useful ahead.

\begin{theorem}
\label{DathJouan} Let $G$ be the two-dimensional solvable Lie group and
consider a linear control system $\Sigma(\mathcal{X},\Delta)$ on $G$ with
$\dim \Delta=1$. Then, $\Sigma(\mathcal{X},\Delta)$ is controllable if and only
if it satisfies the LARC and $\mathfrak{g}=\mathfrak{g}^{0}$.
\end{theorem}

%$\bullet$ If $\rho=\left(\begin{array}{cc}
%0 & -1 \\ 1 & 0
%\end{array}\right)$ it holds that $\rho_{2k\pi}=1$ for any $k\in\Z$ implying that $\widetilde{\Lambda}_{2k\pi}=0$ and consequently that $\psi(D_1)=D_1$, where $D_1$ is the discrete central subgroup that satisfies $E(2)=\widetilde{E(2)}/D_1$. By {\color{red}find reference!!} $\psi$ factors to an automorphims $\phi \in E(2)$ that satisfies $(d\psi)_{\tilde{e}}=(d\phi)_e$ which implies the result for $E(2)$.

%\bigskip

%$\bullet$ If $\rho=\left(\begin{array}{cc}
%0 & 0 \\ 0 & 1
%\end{array}\right)$, the fact that $\widetilde{\Lambda}_0=0$ implies that $\psi(D_2)=D_2$, where $D_2$ is the discrete central subgroup that satisfies $R_2=\widetilde{R_2}/D_2$. As before it holds that $\psi$ factors to an automorphism $\phi$ of $R_2$ that satisfies the lemma, concluding the proof.	

\section{Three-dimensional solvable Lie groups}

This section is devoted to analyze the main ingredients of nonnilpotent,
solvable three-dimensional Lie groups and its corresponding Lie algebras such
as derivations, linear vector fields, invariant vector fields and so on.

Following Chapter 7 of \cite{Oni}, any real three-dimensional nonnilpotent
solvable Lie algebra is isomorphic to one (and only one) of the following Lie algebras:

\begin{itemize}
\item[(i)] the semi-direct product $\mathfrak{r}_{2}=\mathbb{R}\times_{\theta
}\mathbb{R}^{2}$ where $\theta=\left(
\begin{array}
[c]{cc}%
0 & 0\\
0 & 1
\end{array}
\right) ;$

\item[(ii)] the semi-direct product $\mathfrak{r}_{3}=\mathbb{R}\times
_{\theta}\mathbb{R}^{2}$ where $\theta=\left(
\begin{array}
[c]{cc}%
1 & 1\\
0 & 1
\end{array}
\right) ;$

\item[(iii)] the semi-direct product $\mathfrak{r}_{3, \lambda}=\mathbb{R}%
\times_{\theta}\mathbb{R}^{2}$ where $(\lambda \in \mathbb{R}, 0<|\lambda|\leq1)
$ and $\theta=\left(
\begin{array}
[c]{cc}%
1 & 0\\
0 & \lambda
\end{array}
\right) ;$

\item[(iv)] the semi-direct product $\mathfrak{r}^{\prime}_{3, \lambda
}=\mathbb{R}\times_{\theta}\mathbb{R}^{2}$ where $(\lambda \in \mathbb{R},
\lambda \neq0)$ and $\theta=\left(
\begin{array}
[c]{cc}%
\lambda & -1\\
1 & \lambda
\end{array}
\right) ;$

\item[(v)] the semi-direct product $\mathfrak{e}=\mathbb{R}\times_{\theta
}\mathbb{R}^{2}$ where $\theta=\left(
\begin{array}
[c]{cc}%
0 & -1\\
1 & 0
\end{array}
\right) .$
\end{itemize}

The simply connected Lie groups $R_{3}$, $R_{3, \lambda}$, $R^{\prime}_{3,
\lambda}, \widetilde{E}$ and $\widetilde{R_{2}}$ with Lie algebras
$\mathfrak{r}_{3}$, $\mathfrak{r}_{3, \lambda}$, $\mathfrak{r}^{\prime}_{3,
\lambda}, \mathfrak{e}$ and $\mathfrak{r}_{2}$, respectively, are given as the
semi-direct product $\mathbb{R}\times_{\rho}\mathbb{R}^{2}$, where $\rho
_{t}=\mathrm{e}^{t\theta}$.

Associated with $\mathfrak{e}$ we have also the groups $E_{n}:=\widetilde
{E}/D_{n}$ where $D_{n}:=\{(2nk\pi, 0), \;k\in \mathbb{Z}\}$, $n\in \mathbb{N}$.
The group $E_{1}$ is the group of proper motions of $\mathbb{R}^{2}$
(connected component of the whole group of motions of $\mathbb{R}^{2}$) and
$E_{n}$ its $n$-fold covering. Also, if we denote by $v_{k}=(2k\pi,
0)\in \mathbb{R}^{2}$ and consider the discrete central subgroup of $R_{2}$
given by $D=\{(0, v_{k}), k\in \mathbb{Z}\}$ we have the connected Lie group
$R_{2}=\widetilde{R_{2}}/D$.

In the above cases, the canonical projections are given by
\[
\pi_{n}:\widetilde{E}\rightarrow E_{n}, \; \; \pi_{n}(t, v)=(\mathrm{e}%
^{i\frac{t}{n}}, v)\; \; \mbox{ and }\; \; \pi:\widetilde{R_{2}}\rightarrow R_{2},
\; \pi(t, (x, y))=(t, (\mathrm{e}^{ix}, y))
\]
and consequently
\begin{equation}
\label{proj}(d\pi_{n})_{(t, v)}(a, w)=\left( \frac{a}{n}, w\right)
\; \; \mbox{ and }\pi_{*}=\operatorname{id}.
\end{equation}

Moreover, if $G$ is a three-dimensional nonnilpotent, solvable, connected Lie
group and $\widetilde{G}$ its simply connected covering it holds that:

\begin{itemize}
\item[(a)] If $\widetilde{G}=\widetilde{R_{2}}$ then $G=\widetilde{G}$ or
$G=R_{2}$;

\item[(b)] If $\widetilde{G}=R_{3}, R_{3, \lambda}$ or $R^{\prime}_{3,
\lambda}$ then $G=\widetilde{G}$;

\item[(c)] If $\widetilde{G}=\widetilde{E}$ then $G=E_{n}$ for some
$n\in \mathbb{N}$;
\end{itemize}

With exception of the Lie group $\widetilde{E}$, all the three-dimensional
nonnilpotent solvable Lie groups are exponential.

\begin{remark}
We denote by $\mathfrak{aff}(\mathbb{R})$ the only two-dimensional solvable
Lie algebra. The associated connected Lie group is $\mathrm{Aff}%
_{0}(\mathbb{R})$, the connected component of the affine transformations in
$\mathbb{R}$. For the Lie algebra $\mathfrak{r}_{2}$ it holds that
$\mathfrak{r}_{2}=\mathbb{R}\times \mathfrak{aff}(\mathbb{R})$ and consequently
$\widetilde{R_{2}}=\mathbb{R}\times \mathrm{Aff}_{0}(\mathbb{R})$ and
$R_{2}=\mathbb{T}\times \mathrm{Aff}_{0}(\mathbb{R})$.
\end{remark}

In what follows, we analyze the main properties of the above groups. Since we
did not find the next results anywhere we present here their proofs in order
to make the paper self-contained.

For any $s\in \mathbb{R}$ let us define $\Lambda_{s}$ by
\[
\Lambda_{s}:=\left \{
\begin{array}
[c]{cc}%
(\rho_{s}-1)\theta^{-1} & \mbox{ if }\det \theta \neq0\\
\left(
\begin{array}
[c]{cc}%
s & 0\\
0 & \mathrm{e}^{s}-1
\end{array}
\right)  & \mbox{ if }\det \theta=0
\end{array}
\right. .
\]
A simple calculation shows that for any $t, s\in \mathbb{R}$ it holds that
\[
\Lambda_{0}=0, \; \; \; \; \frac{d}{ds}\Lambda_{s}=\rho_{s}, \; \; \; \; \rho
_{s}-\theta \Lambda_{s}=1, \; \; \; \; \rho_{s}\Lambda_{t}=\Lambda_{t}\rho
_{s}\; \; \mbox{ and }\; \; \Lambda_{t}+\rho_{t}\Lambda_{s}=\Lambda_{t+s}.
\]
The above map will be extensively used in the next results.

\begin{proposition}
\label{translations} If $G=\mathbb{R}\times_{\rho}\mathbb{R}^{2}$ then
\[
(dL_{(\tau_{1}, v_{1})})_{(\tau_{2}, v_{2})}(s, w)=(s, \rho_{\tau_{1}%
}w)\; \; \mbox{
and }\; \;(dR_{(\tau_{1}, v_{1})})_{(\tau_{2}, v_{2})}(s, w)=(s, w+s\theta
\rho_{\tau_{2}}v_{1})
\]
and consequently
\[
\exp(s, w)=\left \{
\begin{array}
[c]{cc}%
(0, w) & \mbox{ if }a=0\\
\left( s, \frac{1}{s}\Lambda_{s}w\right) , & \mbox{ if }s\neq0
\end{array}
\right.
\]

\end{proposition}

\begin{proof}
We show the expressions for the left translation since for the right
translation are analogous. The curve $\gamma(t)=(\tau_{2},v_{2})+t(s,w)\in G$
satisfies that $\gamma(0)=(\tau_{2},v_{2})$ and $\gamma^{\prime}(0)=(s,w)$ and
therefore
\[
(dL_{(\tau_{1},v_{1})})_{(\tau_{2},v_{2})}(s,w)=\frac{d}{dt}_{|t=0}%
L_{(\tau_{1},v_{1})}(\gamma(t))=(\tau_{1},v_{1})(\tau_{2}+ts,v_{2}+tw)=
\]%
\[
\frac{d}{dt}_{|t=0}(\tau_{1}+\tau_{2}+ts,v_{1}+\rho_{\tau_{1}}(v_{2}%
+tw))=(s,\rho_{\tau_{1}}w).
\]

To prove the assertion on the exponential, let us consider $(s,w)\in
\mathfrak{g}$ and define the curve
\[
\zeta(t):=\left \{
\begin{array}
[c]{cc}%
(0,tw) & \mbox{ if }s=0\\
\left(  ts,\frac{1}{s}\Lambda_{ts}w\right)  , & \mbox{ if }s\neq0
\end{array}
\right.  .
\]
Since in both cases $\zeta(0)=0$ and
\[
\zeta^{\prime}(t)=\left \{
\begin{array}
[c]{cc}%
\frac{d}{dt}(0,tw)=(0,w)=(dL_{\zeta(t)})_{(0,0)}(0,w) & \mbox{ if }s\neq0\\
\frac{d}{dt}\left(  ts,\frac{1}{s}\Lambda_{ts}w\right)  =(s,\rho
_{ts}w)=(dL_{\zeta(t)})_{(0,0)}(s,w), & \mbox{ if }s\neq0
\end{array}
\right.  .
\]
by unicity we obtain that $\zeta(t)=\exp t(s,w)$ concluding the proof.
\end{proof}

\begin{remark}
The above result and equation (\ref{proj}) imply that the right and left
invariant vector fields on any connected solvable nonnilpotent Lie group $G$
are given respectively by
\[
Y^{L}(\pi(t, v))=(a, \rho_{t} w)\; \; \mbox{ and }\; \; Y^{R}(\pi(t, v))=(a,
w+a\theta v)
\]
where $Y=(a, w)\in \mathfrak{g}$ and $\pi:\widetilde{G}\rightarrow G$ is the
canonical projection.
\end{remark}

Let $\mathcal{X}$ be a linear vector field on $G=\mathbb{R}\times_{\rho
}\mathbb{R}^{2}$ and denote by $\mathcal{D}$ its associated derivation. Since
$\mathcal{D}(\mathbb{R}\times_{\theta}\mathbb{R}^{2})\subset \mathbb{R}^{2}$ we
have a well defined linear map $\mathcal{D}^{*}:\mathbb{R}^{2}\rightarrow
\mathbb{R}^{2}$ satisfying $\mathcal{D}(0, v)=(0, \mathcal{D}^{*} v)$ for
$v\in \mathbb{R}^{2}$. The map $\mathcal{D}^{*}$ satisfies $\mathcal{D}%
^{*}\circ \theta=\theta \circ \mathcal{D}^{*}$. In fact, for any $v\in
\mathbb{R}^{2}$ it holds that
\[
(0, \mathcal{D}^{*}\theta v)=\mathcal{D}(0, \theta v)=\mathcal{D}[(1, 0), (0,
v)]=\underbrace{[\mathcal{D}(1, 0), (0, v)]}_{=0}+[(1, 0), \mathcal{D}(0,
v)]=(0, \theta \mathcal{D}^{*} v)
\]
and therefore $\mathcal{D}^{*}\circ \theta=\theta \circ \mathcal{D}^{*}$. As a
consequence, $\mathcal{D}^{*}\rho_{t}=\rho_{t}\mathcal{D}^{*}$ for any
$t\in \mathbb{R}$.

\begin{proposition}
\label{linear} Let $G$ be a three-dimensional nonnilpotent, solvable,
connected Lie group and denote by $\widetilde{G}$ its simply connected
covering. If $\mathcal{X}$ is a linear vector field on $G$ with associated
derivation $\mathcal{D}$ then
\[
\mathcal{X}(\pi(t, v))=\left( 0, \mathcal{D}^{*} v+\Lambda_{t}\xi \right) ,
\; \; \mbox{ where } (0, \xi)=\mathcal{D}(1, 0)
\]
and $\pi:\widetilde{G}\rightarrow G$ is the canonical projection.
\end{proposition}

\begin{proof}
Let us first consider the case where $G=\mathbb{R}\times_{\rho}\mathbb{R}^{2}%
$. Since
\[
\mathcal{X}(t,v)=\mathcal{X}((t,0)(0,\rho_{-t}v))=(dL_{(t,0)})_{(0,\rho
_{-t}v)}\mathcal{X}(0,\rho_{-t}v)+(dR_{(0,\rho_{-t}v)})_{(t,0)}\mathcal{X}%
(t,0)
\]
it is enough to compute the values of $\mathcal{X}(t,0)$ and $\mathcal{X}%
(0,\rho_{-t}v)$. Moreover, the fact that $(t,0)=\exp(t,0)$ and $(0,w)=\exp
(0,w)$ implies that
\[
\varphi_{s}(0,\rho_{-t}v)=\varphi_{s}(\exp(0,w))=\exp \left(  \mathrm{e}%
^{s\mathcal{D}}(0,\rho_{-t}v)\right)  =\exp(0,\mathrm{e}^{s\mathcal{D}^{\ast}%
}\rho_{-t}v)=(0,\mathrm{e}^{s\mathcal{D}^{\ast}}\rho_{-t}v)
\]
and
\[
\varphi_{s}(t,0)=\varphi_{s}(\exp(t,0))=\exp \left(  \mathrm{e}^{s\mathcal{D}%
}(t,0)\right)  =\exp t\left(  1,\sum_{j\geq0}\frac{s^{j+1}}{(j+1)!}%
(\mathcal{D}^{\ast})^{j}\xi \right)  =\left(  t,\sum_{j\geq0}\frac{s^{j+1}%
}{(j+1)!}\Lambda_{t}\left(  (\mathcal{D}^{\ast})^{j}\xi \right)  \right)  .
\]
Therefore,
\[
\mathcal{X}(0,\rho_{-t}v)=\frac{d}{ds}_{|s=0}\varphi_{s}(0,\rho_{-t}%
v)=(0,\mathcal{D}^{\ast}\rho_{-t}v),\; \; \mbox{ and }\; \; \mathcal{X}%
(t,0)=\frac{d}{ds}_{|s=0}\varphi_{s}(t,0)=(0,\Lambda_{t}\xi).
\]
By using the formulas in Proposition \ref{translations} we get that
\[
\mathcal{X}(t,v)=\left(  0,\mathcal{D}^{\ast}v+\Lambda_{t}\xi \right)
\]
proving the assertion for the simply connected case.

If $G$ is not simply connected, we can consider the linear vector field
$\widetilde{\mathcal{X}}$ on $\widetilde{G}$ that is $\pi$-related to
$\mathcal{X}$. Since $\mathcal{X}$ and $\widetilde{\mathcal{X}}$ have the same
associated derivation $\mathcal{D}$ we have by equation (\ref{proj}) and the
above calculations that
\[
\mathcal{X}(\pi(t, v))=(d\pi)_{(t, v)}\widetilde{\mathcal{X}}(t, v)=(0,
\mathcal{D}^{*} v+\Lambda_{t}\xi)
\]
concluding the proof.
\end{proof}

\begin{remark}
With the notation of the above proposition, it holds that the flow associated
to $\mathcal{X}$ on $\widetilde{G}$ is given by
\begin{equation}
\varphi_{s}(t,v)=(t,\mathrm{e}^{s\mathcal{D}^{\ast}}v+F_{s}\Lambda_{t}%
\xi),\; \; \mbox{ where }\; \;F_{s}=\sum_{j\geq1}\frac{s^{j}(\mathcal{D}^{\ast
})^{j-1}}{j!}.\label{linearflow}%
\end{equation}
In fact, since $\mathcal{D}^{\ast}F_{s}=F_{s}\mathcal{D}^{\ast}=\mathrm{e}%
^{s\mathcal{D}^{\ast}}-1$ and $F_{s}^{\prime}=\mathrm{e}^{s\mathcal{D}^{\ast}%
}$ we get that
\[
\frac{d}{ds}\varphi_{s}(t,v)=\left(  0,\mathcal{D}^{\ast}\mathrm{e}%
^{s\mathcal{D}^{\ast}}v+\mathrm{e}^{s\mathcal{D}^{\ast}}\Lambda_{t}\xi \right)
=\left(  0,\mathcal{D}^{\ast}\Bigl(\underbrace{\mathrm{e}^{s\mathcal{D}^{\ast
}}v+F_{s}\Lambda_{t}\xi}_{=\varphi_{s}(t,v)}\Bigr)+\Lambda_{t}\xi \right)
=\mathcal{X}(\varphi_{s}(t,v)).
\]
In particular, if $\mathcal{D}^{\ast}$ is invertible we obtain $F_{s}%
=(\mathrm{e}^{s\mathcal{D}^{\ast}}-1)(\mathcal{D}^{\ast})^{-1}$.
\end{remark}

The next technical lemma will be useful in the proof of the main results.

\begin{lemma}
\label{isomorphism} Let $G$ be a three-dimensional solvable nonnilpotent
connected Lie group. For any $v_{0}\in \mathbb{R}^{2}$ there exists $\psi
\in \mathrm{Aut}(G)$ satisfying $(d\psi)_{e}(1,v_{0})=(1,0).$
\end{lemma}

\begin{proof}
Let us first consider the simply connected case $\widetilde{G}=\mathbb{R}%
\times_{\rho}\mathbb{R}^{2}$. The map $\psi: \widetilde{G}\rightarrow
\widetilde{G}$ given by
\[
\psi(t, v)=(t, v- \Lambda_{t}v_{0})\; \; \mbox{ has inverse }\; \; \psi^{-1}(t,
v)=(t, v+ \Lambda_{t}v_{0})
\]
and satisfies
\[
\psi(t_{1}, v_{1})\psi(t_{2}, v_{2})=\Bigl(t_{1}, v_{1}-\Lambda_{t_{1}}%
v_{0}\Bigr)\Bigl(t_{2}, v_{2}-\Lambda_{t_{2}}v_{0}\Bigl)=\Bigl(t_{1}+t_{2},
v_{1}-\Lambda_{t_{1}}v_{0}+\rho_{t_{1}}\bigl(v_{2}-\Lambda_{t_{2}}%
v_{0}\bigr)\Bigl)
\]
\[
=\Bigl(t_{1}+t_{2}, v_{1}+\rho_{t_{1}}v_{2}-\Lambda_{t_{1}+t_{2}}%
v_{0}\Bigr)=\psi(t_{1}+t_{2}, v_{1}+\rho_{t_{1}}v_{2})=\psi \bigl((t_{1},
v_{1})(t_{2}, v_{2})\bigr)
\]
implying that $\psi \in \mathrm{Aut}(\widetilde{G})$. Moreover,
\[
(d\psi)_{\tilde{e}}(1, v_{0})=\frac{d}{ds}_{|s=0}\psi(s, sv_{0})=\frac{d}%
{ds}_{|s=0}(s, sv_{0}-\Lambda_{s}v_{0})=(1, v_{0}-\rho_{s}v_{0})|_{s=0}=(1,
0),
\]
proving the result for any $\widetilde{G}$ is simply connected.

If $G$ is not simply connected, one easily shows that the above automorphism
satisfies $\psi(D_{n})=D_{n}$, $n\in \mathbb{N}$ and $\psi(D)=D$, where $D_{n},
n\in \mathbb{N}$ and $D$ are the discrete central subgroups satisfying
$E_{n}=\widetilde{E}/D_{n}$ and $R_{2}=\widetilde{R_{2}}/D$, respectively.
Therefore, $\psi$ factors to an element in $\mathrm{Aut}(G)$ whose
differential coincides with $(d\psi)_{e}$, which proves the result.
\end{proof}

The next lemma states the main properties of derivations in the
three-dimensional Lie algebras under consideration.

\begin{lemma}
\label{derivationprop} Let $\mathfrak{g}=\mathbb{R}\times_{\theta}%
\mathbb{R}^{2}$ and let $\mathcal{D}:\mathfrak{g}\rightarrow \mathfrak{g}$ be a
derivation. It holds:

\begin{itemize}
\item[1.] $\dim \mathfrak{g}^{0}=1$ if and only if $\mathcal{D}^{*}$ is invertible;

\item[2.] $\mathbb{R}^{2}\subset \ker \mathcal{D}$ then $\mathcal{D}$ is nilpotent.

\item[3.] Any derivation on $\mathfrak{r}_{3}$ or on $\mathfrak{r}_{3,
\lambda} $ with $0<|\lambda|<1$ has only real eigenvalues. Therefore, if
$\mathcal{D}^{*}$ admits a pair of complex eigenvalues we must have
$\lambda=1$;

\item[4.] Any derivation $\mathcal{D}$ of $\mathfrak{r}^{\prime}_{3, \lambda}$
or of $\mathfrak{e}$ satisfies $\mathcal{D}^{*}=\left(
\begin{array}
[c]{cc}%
\alpha & -\beta \\
\beta & \alpha
\end{array}
\right) .$
\end{itemize}
\end{lemma}

\begin{proof}
1. Since $\mathfrak{g}=\mathfrak{g}^{+}\oplus \mathfrak{g}^{0}\oplus
\mathfrak{g}^{-} $, if $\dim \mathfrak{g}^{0}=1$ we have that $\dim
(\mathfrak{g}^{+}\oplus \mathfrak{g}^{-})=2$. By Lemma \ref{nilradical} it
holds that $\mathfrak{g}^{+}\oplus \mathfrak{g}^{-}\subset \mathbb{R}^{2}$ and
consequently $\mathbb{R}^{2}=\mathfrak{g}^{+}\oplus \mathfrak{g}^{-}$ implying
that $\mathcal{D}^{*}$ is invertible. Reciprocally, since $\mathcal{D}$
invertible implies $\mathfrak{g}$ nilpotent we must have necessarily that
$\dim \mathfrak{g}^{0}\geq1$. The fact that the eigenvalues of $\mathcal{D}%
^{*}$ are also eigenvalues of $\mathcal{D}$ implies then that $\dim
\mathfrak{g}^{0}=1$ if $\mathcal{D}^{*}$ is invertible.

2. In fact, if $\mathbb{R}^{2}\subset \ker \mathcal{D}$ then necessarily
$\mathcal{D}^{*}\equiv0$ and consequently $\mathcal{D}^{2}(\mathfrak{g}%
)\subset \mathcal{D}^{*}\mathbb{R}^{2}=\{0\}$ showing that $\mathcal{D}$ is nilpotent.

3. Since all the nonzero eigenvalues of $\mathcal{D}$ are also eigenvalues of
$\mathcal{D}^{*}$ and $\mathcal{D}^{*}$ commutes with $\theta$ it holds that
\[
\mathcal{D}^{*}=\left(
\begin{array}
[c]{cc}%
\alpha & \beta \\
0 & \alpha
\end{array}
\right) \; \; \; \mbox{ and }\; \; \; \mathcal{D}^{*}=\left(
\begin{array}
[c]{cc}%
\alpha & 0\\
0 & \beta
\end{array}
\right)
\]
when $\mathcal{D}$ is a derivation of $\mathfrak{r}_{3}$ and $\mathfrak{r}_{3,
\lambda}$, $|\lambda|\in(0, 1)$, respectively.

4. It follows directly from item 3.

5. It follows from the fact that $\mathcal{D}^{*}$ and $\theta$ commutes.
\end{proof}

We have also the following.

\begin{lemma}
\label{subalgebra} The only two-dimensional Lie subalgebra of $\mathfrak{e}$
or $\mathfrak{r}^{\prime}_{3, \lambda}$ is the nilradical $\mathbb{R}^{2}$.
\end{lemma}

\begin{proof}
In fact, if $\mathfrak{h}$ is a two-dimensional Lie subalgebra of
$\mathfrak{g}$, where $\mathfrak{g}\in \{ \mathfrak{e},\mathfrak{r}_{3,\lambda
}^{\prime}\}$ then necessarily $\dim(\mathfrak{h}\cap \mathbb{R}^{2})\geq1$ and
consequently $(0,w)\in \mathfrak{h}$. If $(t,w^{\prime})\in \mathfrak{h}$ is
such that $\mathfrak{h}=\mathrm{span}\{(t,w^{\prime}),(0,w)\}$ then
\[
\lbrack(t,w^{\prime}),(0,w)]=(0,t\theta w)\in \mathfrak{h}%
\]
and hence $t\theta w=t^{\prime}w$. Since $\theta$ does not admits any
invariant one-dimensional subspace we must have that $t=0$ and consequently
$\mathfrak{h}=\mathbb{R}^{2}$ as desired.
\end{proof}

Concerning linear systems on nonnilpotent, solvable three-dimensional Lie
groups, the next result states that we can concentrate our studies to two
specific kind of systems.

\begin{proposition}
\label{conjugation} Any linear control system $\Sigma(\mathcal{X}, \Delta)$ on
a three-dimensional, solvable, connected, nonnilpotent Lie group $G$ that
satisfies the LARC is equivalent to one of the following linear systems:
\[
\dot{g}=\mathcal{X}(g)+uY^{1}(g)\; \; \; \mbox{ or }\; \; \; \dot{g}=\mathcal{X}%
(g)+u_{1}Y^{1}(g)+u_{2}Y^{2}(g),
\]
where $Y_{1}=(1, 0)$ and $Y_{2}=(0, w)$, for some $w\in \mathbb{R}^{2}$.
\end{proposition}

\begin{proof}
We only have to analyze the cases where $\dim \mathfrak{h}=1 \mbox{ or }2$. For
both cases, the fact that $\Sigma(\mathcal{X}, \Delta)$ satisfies the LARC
implies that $\Delta \not \subset \mathbb{R}^{2}$ and hence:

\begin{itemize}
\item[1.] If $\dim \Delta=1$ then $\Delta=\mathrm{span}\{(1, v_{0})\}$ for some
$v_{0}\in \mathbb{R}^{2}$.

\item[2.] If $\dim \Delta=2$ then $\Delta=\mathrm{span}\{(1, v_{0}), (0, w)\}$
for some $v_{0}, w\in \mathbb{R}^{2}$.
\end{itemize}

By considering the isomorphism $\psi$ given by Proposition \ref{isomorphism}
we get that $\Sigma(\mathcal{X},\Delta)$ is equivalent to the system
$\Sigma(\mathcal{X}_{\psi},\Delta_{\psi})$ that has necessarily the form
\[
\dot{g}=\mathcal{X}(g)+uY^{1}(g)\; \; \; \mbox{ or }\; \; \; \dot{g}=\mathcal{X}%
(g)+u_{1}Y^{1}(g)+u_{2}Y^{2}(g),
\]
for $Y_{1}=(1,0)$ and $Y_{2}=(0,w)$, for some $w\in \mathbb{R}^{2},$ concluding
the proof.
\end{proof}

\subsection{Homogeneous space}

In this section we analyze homogeneous space of the three-dimensional
nonnilpotent solvable Lie groups which will be used in the sections ahead. Our
particular interest are the projections of linear and invariant vector fields
to these homogenous spaces.

\subsubsection*{$\bullet$ $G=\mathbb{R}\times_{\rho}\mathbb{R}^{2}\;
\mbox{ and }\; \mathcal{D}^{*}$ is invertible}

For this case, we will consider the group of the singularities of
$\mathcal{X}$ given by $F=\{(t,v)\in G;\; \mathcal{X}(t,v)=0\}$. Since
$\mathcal{D}^{\ast}$ is invertible, it holds that $F$ is a one-dimensional
closed Lie subgroup of $G$. Moreover, by Proposition \ref{linear} it holds
that
\[
(t,v)\in F\; \; \iff \; \; \mathcal{D}^{\ast}v=-\Lambda_{t}\xi,\; \; \mbox{
where }\;(0,\xi)=\mathcal{D}(1,0).
\]
Hence,
\[
F(t_{1},v_{1})=F(t_{2},v_{2})\; \; \iff \; \; \rho_{-t_{1}}\mathcal{D}^{\ast}%
v_{1}-\Lambda_{-t_{1}}\xi=\rho_{-t_{2}}\mathcal{D}^{\ast}v_{2}-\Lambda
_{-t_{2}}\xi.
\]
Consequently, we can identify the homogeneous space $F\setminus G$ with
$\mathbb{R}^{2}$ by using the map
\[
F\cdot(t,v)\in F\setminus G\mapsto \rho_{-t}\mathcal{D}^{\ast}v-\Lambda_{-t}%
\xi \in \mathbb{R}^{2}.
\]
Under this identification, the projection $\pi:G\rightarrow F\setminus G$ is
given by
\[
\pi(t,v)=\rho_{-t}\mathcal{D}^{\ast}v-\Lambda_{-t}\xi \; \; \mbox{ and its
differential by}\; \;(d\pi)_{(t,v)}(a,w)=-a\rho_{-t}(\mathcal{D}^{\ast}\theta
v-\xi)+\rho_{-t}\mathcal{D}^{\ast}w.
\]
Therefore, we have that
\begin{equation}
(d\pi)_{(t,v)}\mathcal{X}(t,v)=\mathcal{D}^{\ast}\pi
(t,v)\; \; \mbox{ and }\; \;(d\pi)_{(t,v)}Y^{L}(t,v)=-a(\theta \pi(t,v)-\xi
)+\mathcal{D}^{\ast}w\label{projection}%
\end{equation}
where $Y=(a,w)\in \mathfrak{g}$.

\begin{remark}
It is not hard to see that if we consider the above setup on $E_{n}%
,\,n\in \mathbb{N},$ both, the homogeneous space $F\setminus G$ and the
projection $\pi:G\rightarrow F\setminus G$ have the same expression.
\end{remark}

\subsubsection*{$\bullet$ $G=\mathbb{R}\times_{\rho}\mathbb{R}^{2}\;
\mbox{ and }\; \mathcal{D}^{*}$ is identically zero}

Let $w_{0}\in \mathbb{R}^{2}$ and consider the one-parameter subgroup of
$(0,w_{0})$ given by
\[
H_{w_{0}}=\{ \exp s(0,w_{0}),\;s\in \mathbb{R}\}=\{(0,sw_{0}),\; \;s\in
\mathbb{R}\}.
\]
It holds that
\[
H_{w_{0}}(t_{1},v_{1})=H_{w_{0}}(t_{2},v_{2})\; \; \iff \; \;t_{1}=t_{2}%
\; \; \mbox{ and }v_{2}-v_{1}\in \mathrm{span}\{w_{0}\}.
\]
If we consider $v_{0}\in \mathbb{R}^{2}$ such that $\langle w_{0},v_{0}%
\rangle=0$ then $v_{2}-v_{1}\in \mathrm{span}\{w_{0}\} \iff \langle v_{1}%
,v_{0}\rangle=\langle v_{2},v_{0}\rangle$ and consequently we can identify the
homogeneous space $H_{w_{0}}\setminus G$ with $\mathbb{R}^{2}$ using the map
\[
H_{w_{0}}\cdot(t,v)\in H_{w_{0}}\setminus G\mapsto(t,\langle v,v_{0}%
\rangle)\in \mathbb{R}^{2}.
\]
Under this identification, the projection $\pi:G\rightarrow H_{w_{0}}\setminus
G$ is given by
\[
\pi(t,v)=(t,\langle v,v_{0}\rangle
)\; \; \mbox{ and since it is linear }\; \;(d\pi)_{(t,v)}=\pi.
\]
We obtain,
\begin{equation}
(d\pi)_{(t,v)}\mathcal{X}(t,v)=(0,\langle \Lambda_{t}\xi,v_{0}\rangle
)\; \; \mbox{ and }\; \;(d\pi)_{(t,v)}Y^{L}(t,v)=(a,\langle \rho_{t}w,v_{0}%
\rangle),\label{projection2}%
\end{equation}
where $Y=(a,w)\in \mathfrak{g}$.

\subsubsection*{$\bullet$ $G=R_{2}$ and $\mathcal{D}^{*}$ is identically zero}

Let $w_{0}\in \mathbb{R}^{2}$ and assume that $w_{0}=(\alpha,\beta)$ with
$\alpha,\beta \in \mathbb{R}^{\ast}$. The one-parameter subgroup of $(0,w_{0})$
is given by
\[
H_{w_{0}}=\{ \exp s(0,w_{0}),\;s\in \mathbb{R}\}=\{(0,(\mathrm{e}^{is\alpha
},s\beta)),\; \;s\in \mathbb{R}\}.
\]
Then
\[
H_{w_{0}}(t_{1},v_{1})=H_{w_{0}}(t_{2},v_{2})\; \; \iff \; \;t_{1}=t_{2}%
\; \;,\mathrm{e}^{i(x_{1}+s\alpha)}=\mathrm{e}^{ix_{2}}%
\; \; \mbox{ and }\; \;y_{1}+s\beta=y_{2}%
\]%
\[
\iff t_{1}=t_{2}\; \; \mbox{ and }\; \; \mathrm{e}^{i(x_{1}-\frac{\alpha}{\beta
}y_{1})}=\mathrm{e}^{i(x_{2}-\frac{\alpha}{\beta}y_{2})}%
\]
If we consider $v_{0}\in \mathbb{R}^{2}$ such that $\langle w_{0},v_{0}%
\rangle=0$ then $v_{0}=(\beta,-\alpha).$ Consequently, we can identify the
homogeneous space $H_{w_{0}}\setminus G$ with $\mathbb{R}\times \mathbb{T}$
using the map
\[
H_{w_{0}}\cdot(t,v)\in H_{w_{0}}\setminus G\mapsto \left(  t,\mathrm{e}%
^{i\beta^{-1}\langle v,v_{0}\rangle}\right)  \in \mathbb{R}\times \mathbb{T}.
\]
Under this identification, the projection $\pi:G\rightarrow H_{w_{0}}\setminus
G$ is given by
\[
\pi(t,v)=\left(  t,\mathrm{e}^{i\beta^{-1}\langle v,v_{0}\rangle}\right)
\; \; \mbox{ and its differential by }\; \;(d\pi)_{(t,v)}(a,w)=(a,\beta
^{-1}\langle v,v_{0}\rangle).
\]
In particular, we get
\begin{equation}
(d\pi)_{(t,v)}\mathcal{X}(t,v)=(0,\beta^{-1}\langle \Lambda_{t}\xi,v_{0}%
\rangle)\; \; \mbox{ and }\; \;(d\pi)_{(t,v)}Y^{L}(t,v)=(a,\beta^{-1}\langle
\rho_{t}w,v_{0}\rangle),\label{projection3}%
\end{equation}
where $Y=(a,w)\in \mathfrak{g}$.

\section{Controllability}

In this section we analyze the controllability property of linear control
systems on three-dimensional nonnilpotent solvable Lie groups. Since the LARC
is a necessary condition for controllability our work is reduced to the
analysis of linear systems on $\Sigma(\mathcal{X},\Delta)$, where $\dim
\Delta=1$ or $2$.

\subsection{The one-dimensional case}

In this section we analyze the case where $\dim \Delta=1$. In this context, the
next theorem summarizes the controllability of linear control systems on the
different classes of three-dimensional nonnilpotent, solvable Lie groups. Its
proof will be divided in several propositions.

\begin{theorem}
Let $\Sigma(\mathcal{X},\Delta)$ be a linear system on a three-dimensional
nonnilpotent solvable Lie group $G$ that satisfies the LARC and $\dim \Delta
=1$. It holds:

\begin{itemize}
\item[1.] If $G=R_{2}$: $\Sigma(\mathcal{X}, \Delta)$ is controllable if and
only if $\mathfrak{g}^{0}\simeq \mathfrak{aff}(\mathbb{R})$ or $\mathfrak{g}%
=\mathfrak{g}^{0}$;

\item[2.] If $G=\widetilde{R_{2}}$: $\Sigma(\mathcal{X}, \Delta)$ is
controllable if and only if $\mathfrak{g}^{0}=\mathfrak{aff}(\mathbb{R})$;

\item[3.] If $G=E_{n}, \widetilde{E}$ or $R_{3}$: $\Sigma(\mathcal{X},
\Delta)$ is controllable if and only if $\mathfrak{g}=\mathfrak{g}^{0}$ and
$\mathcal{D}^{*}\not \equiv 0$;

\item[4.] If $G=R_{3, \lambda}$: $\Sigma(\mathcal{X}, \Delta)$ is controllable
if and only if $\lambda=1$ and $\mathcal{D}^{*}$ has a pair of complex eigenvalues

\item[5.] If $G=R^{\prime}_{3, \lambda}$: $\Sigma(\mathcal{X}, \Delta)$ is controllable.
\end{itemize}
\end{theorem}

In the sequel, we prove the above theorem analyzing case by case.

\subsubsection{The case $G=R_{2}$ or $G=\widetilde{R_{2}}$.}

\begin{proposition}
Let $\Sigma(\mathcal{X}, \Delta)$ be a linear control system on $G$. Then,
$\Sigma(\mathcal{X}, \Delta)$ is controllable if and only if it satisfies the
LARC and

\begin{itemize}
\item[(i)] $G=\widetilde{R_{2}}$ and $\mathfrak{g}^{0}\simeq \mathfrak{aff}%
(\mathbb{R})$;

\item[(ii)] $G=R_{2}$ and $\mathfrak{g}^{0}\simeq \mathfrak{aff}(\mathbb{R})$
or $\mathfrak{g}=\mathfrak{g}^{0}$;
\end{itemize}
\end{proposition}

\begin{proof}
Let us start by proving the following facts:

a) \textit{If $\Sigma(\widetilde{\mathcal{X}},\Delta)$ is a linear control
system on $\widetilde{R_{2}}$ which satisfies the LARC and $\mathfrak{g}%
^{0}\simeq \mathfrak{aff}(\mathbb{R})$ then, it is controllable.}

In fact, if $\mathfrak{g}^{0}\simeq \mathfrak{aff}(\mathbb{R})$ then
necessarily $\mathfrak{g}=\mathbb{R}e_{1}\oplus \mathfrak{g}^{0}$ implying that
$\mathfrak{g}^{0}$ is an ideal of $\mathfrak{g}$. Hence, $\mathcal{D}%
e_{1}=\lambda e_{1}$ with $\lambda \in \mathbb{R}^{\ast}$ and $\mathcal{D}%
e_{2}=0$. Moreover, if $\Sigma(\widetilde{\mathcal{X}},\Delta)$ satisfies the
LARC and $(0,\xi)=\mathcal{D}(1,0)$ we must have $\xi=ae_{1}+be_{2}$ with
$a,b\in \mathbb{R}^{\ast}$. Thus, $\mathcal{D}^{2}(1,0)=(0,a\lambda e_{1})$
showing that $\{(1,0),\mathcal{D}(1,0),\mathcal{D}^{2}(1,0)\}$ is a basis for
$\mathfrak{g}$ and therefore that $\Sigma(\widetilde{\mathcal{X}},\Delta)$
satisfies the ad-rank condition.

Since $\mathfrak{g}^{0}$ is an ideal of $\mathfrak{g}$ we can consider the
induced linear system on $G/\widetilde{G}^{0}\simeq \mathbb{R}$ that is
controllable by the ad-rank condition (see Theorem 3 of \cite{Sontag}).
Moreover, the ad-rank condition implies also by Theorem \ref{meuteorema} that
$G^{0}\subset \mathcal{A}\cap \mathcal{A}^{*}$ which by Lemma \ref{translation}
gives us the controllability $\Sigma(\widetilde{\mathcal{X}}, \Delta)$.

b) \textit{If $\Sigma(\mathcal{X}, \Delta)$ is a linear system on $R_{2}$
which satisfies the LARC and $\mathfrak{g}^{0}\simeq \mathfrak{aff}%
(\mathbb{R})$ then $\Sigma$ cannot be controllable.}

In fact, for such system we have the following possibilities:

\begin{itemize}
\item $\dim \mathfrak{g}^{0}=1$: By Lemma \ref{derivationprop} the linear map
$\mathcal{D}^{*}$ is invertible and consequently, the induced system on
$G/\mathbb{T}\simeq \mathrm{Aff}_{0}(\mathbb{R})$ cannot be controllable by
Theorem \ref{DathJouan}.

\item $\dim \mathfrak{g}^{0}=2$: In this case $\mathfrak{g}^{0}$ is Abelian and
necessarily $e_{1}\in \mathfrak{g}^{0}$ since otherwise the whole Lie algebra
$\mathfrak{g}$ would be Abelian. Moreover, since $\mathcal{D}e_{3}=0$ would
imply $\mathfrak{g}=\mathfrak{g}^{0}$ we must have $\mathcal{D}e_{3}=\mu
e_{3}$ for some $\mu \in \mathbb{R}^{\ast}$. As in the previous item, we do not
have controllability of the induced system on $G/\mathbb{T}\simeq
\mathrm{Aff}_{0}(\mathbb{R})$.
\end{itemize}

Let us now consider $\pi$-related linear control systems $\Sigma
(\widetilde{\mathcal{X}}, \Delta)$ and $\Sigma(\mathcal{X}, \Delta)$ on
$\widetilde{R_{2}}$ and $R_{2}$ respectively, where $\pi:\widetilde{R_{2}%
}\rightarrow R_{2}$ is the canonical projection. Note that $\Sigma
(\widetilde{\mathcal{X}}, \Delta)$ satisfies the LARC if and only if
$\Sigma(\mathcal{X}, \Delta)$ also satisfies it.

(i) By item a) above, if $\mathfrak{g}^{0}\simeq \mathfrak{aff}(\mathbb{R})$
and $\Sigma(\widetilde{\mathcal{X}},\Delta)$ satisfies the LARC then it is
controllable. Reciprocally, if $\dim \mathfrak{g}^{0}=1$ or $\dim
\mathfrak{g}^{0}=2$ and $\mathfrak{g}^{0}$ is Abelian, $\Sigma(\mathcal{X}%
,\Delta)$ is not controllable by item b) and consequently $\Sigma
(\widetilde{\mathcal{X}},\Delta)$ cannot be controllable. The only remaining
possibility is $\mathfrak{g}=\mathfrak{g}^{0}$. Since in this case we
necessarily have $\mathcal{D}^{\ast}\equiv0$, for any $w_{0}\in \mathbb{R}%
^{2},$ we can consider the induced system on $\widetilde{H}_{w_{0}}%
\setminus \widetilde{G}$ for $w_{0}\in \mathbb{R}^{2}$ given by
\begin{equation}
\left \{
\begin{array}
[c]{l}%
\label{4}\dot{t}=u\\
\dot{z}=\langle \Lambda_{t}\xi,v_{0}\rangle
\end{array}
\right.  ,\; \; \mbox{ where }\; \; \langle w_{0},v_{0}\rangle=0.
\end{equation}
By assuming that $\Sigma(\widetilde{\mathcal{X}},\Delta)$ satisfies the LARC
it holds that $\xi=(\xi_{1},\xi_{2})$ with $\xi_{1},\xi_{2}\in \mathbb{R}%
^{\ast}$. By considering $w_{0}=(-\xi_{2}^{-1},\xi_{1}^{-1})$ we get that
\[
\dot{z}=\langle \Lambda_{t}\xi,v_{0}\rangle=1+t-\mathrm{e}^{t}\leq0
\]
implying that (\ref{4}) cannot be controllable since the region $\mathcal{C}%
=\{(z,t);\;z\leq0\}$ is invariant by its solutions. Consequently,
$\Sigma(\widetilde{\mathcal{X}},\Delta)$ cannot be controllable concluding the
proof of case $(i)$.

(ii) By item a) if $\mathfrak{g}^{0}\simeq \mathfrak{aff}_{0}(\mathbb{R})$ and
$\Sigma(\mathcal{X},\Delta)$ satisfies the LARC it follows that $\Sigma
(\widetilde{\mathcal{X}},\Delta)$ is controllable and consequently
$\Sigma(\mathcal{X},\Delta)$ is controllable. By item b) $\Sigma
(\mathcal{X},\Delta)$ cannot be controllable when $\dim \mathfrak{g}^{0}=1$ or
when $\dim \mathfrak{g}^{0}=2$ and $\mathfrak{g}^{0}$ is Abelian. Therefore, we
only have to show that $\mathfrak{g}=\mathfrak{g}^{0}$ together with the LARC
implies the controllability of $\Sigma(\mathcal{X},\Delta)$.

In this case, for any $w_{0}=(\alpha, \beta)\in \mathbb{R}^{2}$ with $\alpha,
\beta \in \mathbb{R}^{*}$ we have the induced system on $H_{w_{0}}\setminus
G\simeq \mathbb{R}\times \mathbb{T}$ given by
\begin{equation}
\left \{
\begin{array}
[c]{l}%
\label{8} \dot{t}=u\\
\dot{z}=\beta^{-1}\langle \Lambda_{t}\xi, v_{0}\rangle
\end{array}
\right. , \; \; \mbox{ where }\; \; \langle w_{0}, v_{0}\rangle=0.
\end{equation}

Since $\Sigma(\mathcal{X},\Delta)$ satisfies the LARC we must have that
$\mathcal{D}(1,0)=(0,\xi)$ with $\xi=(\xi_{1},\xi_{2})$ with $\xi_{1},\xi
_{2}\in \mathbb{R}^{\ast}$ and hence, by considering $w_{0}=\xi$ we get that
$v_{0}=(-\xi_{2},\xi)$ and
\[
\dot{z}=\xi_{2}^{-1}\langle \Lambda_{t}\xi,v_{0}\rangle=\xi_{1}(\mathrm{e}%
^{t}-t-1).
\]

For such system we have the following equalities
\[%
\begin{array}
[c]{ll}%
\pi_{1}\left(  \varphi(s,(t,z),u)\right)  =t+us, & \mbox{ if }u\equiv
\mbox{cte}\\
\varphi(s,(t,z),u)=(t,z\cdot \mathrm{e}^{is\xi_{1}(\mathrm{e}^{t}-t-1)}) &
\mbox{ if }u\equiv0
\end{array}
\]
where $\pi_{1}:\mathbb{R}\times \mathbb{T}\rightarrow \mathbb{R}$ is the
projection onto the first coordinate. For any given $q=(a,z)\in \mathbb{R}%
\times \mathbb{T}$ with $a\neq0$ we construct a trajectory from $(0,(1,0))$ to
$q$ as follows:

\begin{itemize}
\item[1.] We go from $(0, (1, 0))$ to a point $q^{\prime}=(a, z^{\prime})$ by
using a constant control function;

\item[2.] By ``switching off" the control we can go from $q^{\prime}=(a,
z^{\prime})$ to $q=(a, z)$, since $z^{\prime}\cdot \mathrm{e}^{is\xi
_{1}(\mathrm{e}^{a}-a-1)}$ covers the whole $\mathbb{T}$ if $a\neq0$.
\end{itemize}

\begin{figure}[h]
\begin{center}
\includegraphics[scale=.5]{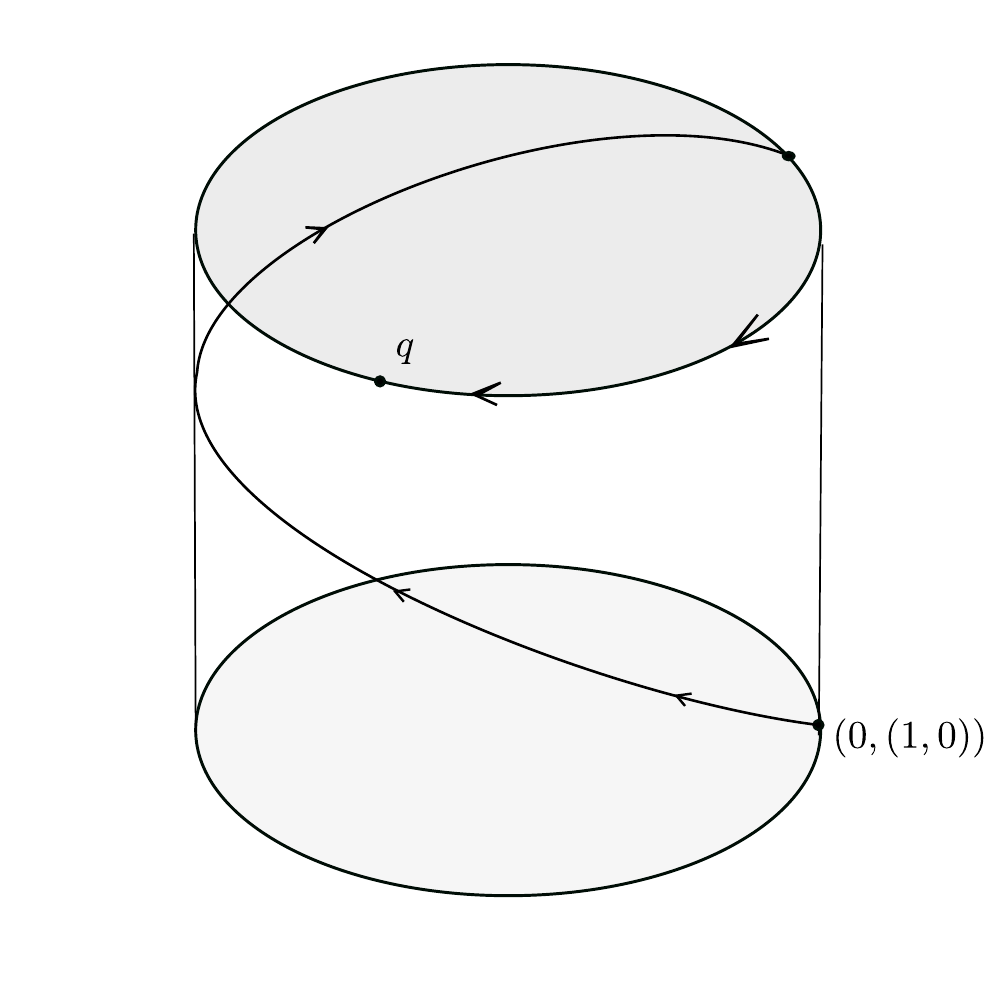}
\end{center}
\caption{Solutions connecting $(0, (1,0))$ to $q\in \mathbb{R}\times \mathbb{T}
$.}%
\label{fig4}%
\end{figure}

Since $\Sigma(\mathcal{X},\Delta)$ projects to the system (\ref{8}), the
projection of $\mathcal{A}$ and $\mathcal{A}^{\ast}$ are dense in $H_{\xi
}\setminus G$ and consequently $H_{\xi}\cdot \mathcal{A}$ and $H_{\xi}%
\cdot \mathcal{A}^{\ast}$ are dense in $G$.

On the other hand, for any $t, s\in \mathbb{R}$ it holds that
\[
\varphi_{t}(s(1, 0))=\varphi_{t}(\exp s(1, 0))=\exp s(\mathrm{e}^{\mathcal{D}%
}(1, 0))=\exp s(1, t \xi)=(s, t\Lambda_{s}\xi).
\]
In particular, for any $r>0$ we can consider $t=r/|s|$ and so $\left( s,
\frac{rs}{|s|}\frac{1}{s}\Lambda_{s}v_{0}\right) =\varphi_{t}(s(1,
0))\in \mathcal{A} $ for any $s\in \mathbb{R}$. By considering $s\rightarrow0$
from both sides we get that $(0, \pm r\xi)\in \overline{\mathcal{A}}$ and since
$r>0$ was arbitrary we conclude that $H_{\xi} \subset \overline{\mathcal{A}}$
which by Lemma \ref{translation} implies also that $H_{\xi} \subset
\overline{\mathcal{A}^{*}}$ and consequently
\[
G=\overline{H_{\xi}\cdot \mathcal{A}}\subset \overline{\mathcal{A}}\; \; \mbox{
and }G=\overline{H_{\xi}\cdot \mathcal{A}^{*}}\subset \overline{\mathcal{A}^{*}%
}.
\]
Since $\Sigma(\mathcal{X}, \Delta)$ satisfies the LARC $G=\overline
{\mathcal{A}}\cap \overline{\mathcal{A}^{*}}$ implies $G=\mathcal{A}%
\cap \mathcal{A}^{*}$ concluding the proof.
\end{proof}

\subsubsection{The case $G=E_{n}$ or $G=\widetilde{E_{2}}$.}

\begin{proposition}
A linear control system $\Sigma(\mathcal{X}, \Delta)$ on $G$ is controllable
if and only if it satisfies the LARC and $\mathcal{D}$ has a pair of purely
imaginary eigenvalues.
\end{proposition}

\begin{proof}
If $\Sigma(\mathcal{X}, \Delta)$ satisfies the LARC then $\mathcal{D}(1,
0)=(0, \xi)\neq0$. If we also assume that $\mathcal{D}$ has a pair of purely
imaginary eigenvalues, then $\mathfrak{g}=\mathfrak{g}^{0}$ and $\{ \xi,
\mathcal{D}^{*} \xi \}$ is linearly independent implying that $\Sigma
(\mathcal{X}, \Delta)$ satisfies the ad-rank condition, which by Theorem
\ref{meuteorema} implies its controllability.

Reciprocally, let us then assume that $\mathcal{D}$ does not admit a pair of
purely imaginary eigenvalues. By Proposition \ref{derivationprop} the
eigenvalues of $\mathcal{D}^{\ast}$ are of the form $\alpha \pm i\beta$. If
$\Sigma(\mathcal{X},\Delta)$ satisfies the LARC then we can assume w.l.o.g.
that $(1,0)\in \Delta$ and we have the following possibilities:

$\bullet$ If $\alpha \neq0$ we have that $\mathcal{D}^{\ast}$ is invertible and
so, the induced system on the homogeneous space $F\setminus G$ is given by
$\dot{v}=\mathcal{D}^{\ast}v-u(\theta v-\xi)$ which in coordinates reads as
\begin{equation}
\left \{
\begin{array}
[c]{l}%
\dot{x}=\alpha x-\beta y+u(y+\xi_{1})\\
\dot{y}=\beta x+\alpha y-u(x-\xi_{2})
\end{array}
\right.  \; \; \; \mbox{ where }\; \; \; \xi=(\xi_{1},\xi_{2}).\label{1}%
\end{equation}
Using the fact that $\alpha \neq0$, a simple calculation shows that $(x-\xi
_{2})\dot{x}+(y+\xi_{1})\dot{y}=\alpha \,K(x,y)$ where
\[
K(x,y)=\left[  \left(  x+\frac{\beta \xi_{1}-\alpha \xi_{2}}{2\alpha}\right)
^{2}+\left(  y+\frac{\beta \xi_{2}+\alpha \xi_{1}}{2\alpha}\right)  ^{2}%
-\frac{(\beta^{2}+\alpha^{2})|\xi|^{2}}{4\alpha^{2}}\right]  .
\]
Therefore, if $\alpha>0$ the solutions of (\ref{1}) let the exterior
$\mathcal{C}$ of any circumference with center at $\xi$ and radius
$R>\frac{((\alpha+\beta)^{2}+2\alpha^{2})|\xi|^{2}}{2\alpha^{2}}$ invariant
(Figure \ref{fig2}). Analogously, if $\alpha<0$ the interior of any such
circumference is invariant by the solutions of (\ref{1}). Therefore, (\ref{1})
cannot be controllable and consequently $\Sigma(\mathcal{X},\Delta)$ cannot be controllable.

\begin{figure}[h]
\begin{center}
\includegraphics[scale=.8]{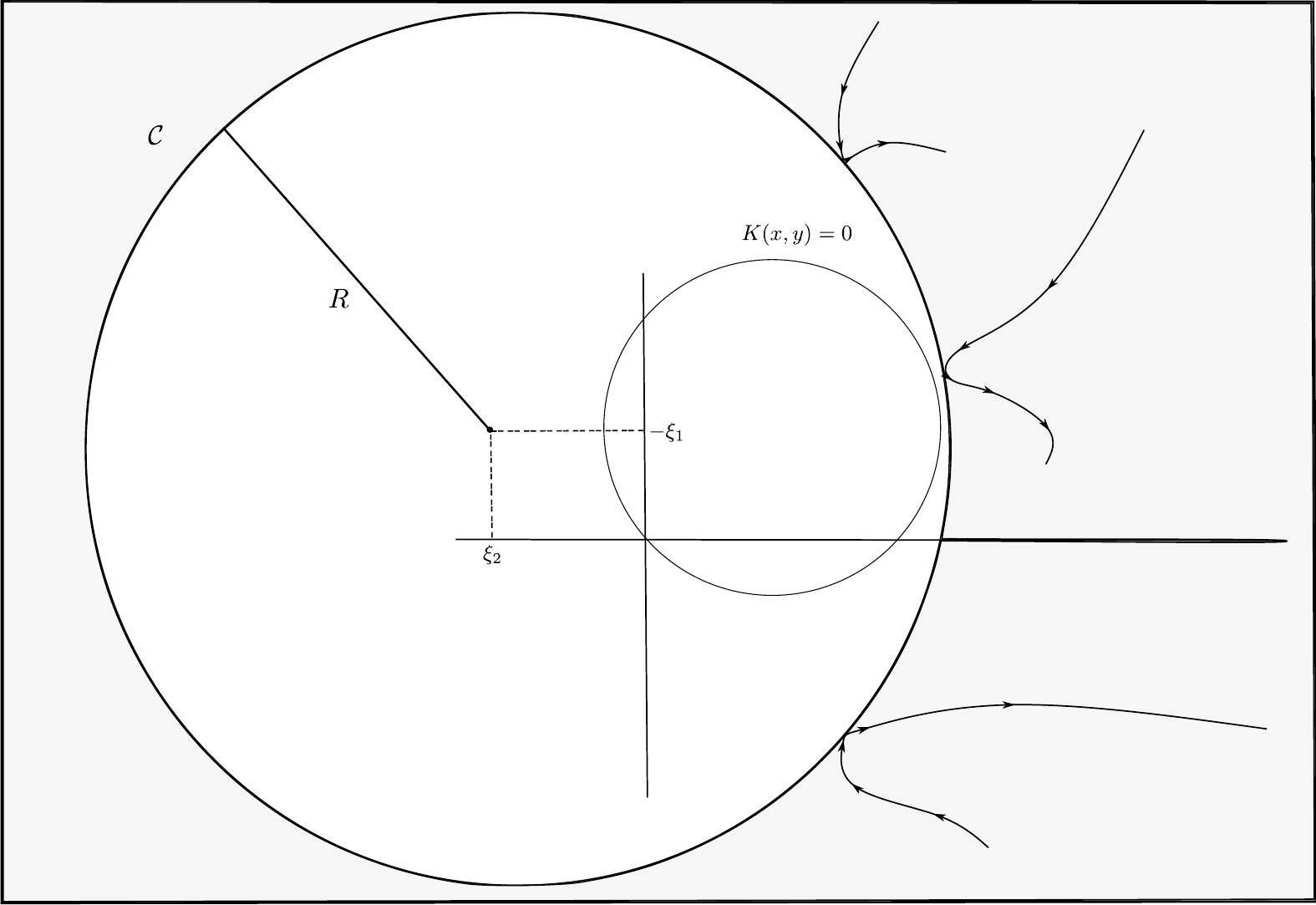}
\end{center}
\caption{Invariant region $\mathcal{C}$ for $\alpha>0$.}%
\label{fig2}%
\end{figure}

\bigskip

$\bullet$ If $\alpha=\beta= 0$ we can consider the induced system on $H_{\xi
}\setminus G$. By \ref{projection2} such system is given by
\begin{equation}
\label{2}\left \{
\begin{array}
[c]{l}%
\dot{t}=u\\
\dot{x}=\langle \Lambda_{t}\xi, \theta \xi \rangle
\end{array}
\right.  .
\end{equation}
However, since $\theta^{-1}=-\theta$ and $\| \rho_{t}\|=1$ we have that
\[
\langle \Lambda_{t} \xi, \theta \xi \rangle=\langle(1-\rho_{t})\theta \xi
,\theta \xi \rangle=|\xi|^{2}-\langle \rho_{t}\xi, \xi \rangle \geq|\xi|^{2}%
-\| \rho_{t}\||\xi|^{2}\geq0
\]
implying that $\dot{x}\geq0$ and hence that (\ref{2}) cannot be controllable.

Therefore, the condition on $\mathcal{D}$ admitting a pair of purely imaginary
eigenvalues is a necessary condition for the controllability of $\Sigma$
concluding the proof.
\end{proof}

\subsubsection{The case $G=R_{3}, \;R_{3, \lambda}$ or $R^{\prime}_{3,
\lambda}$.}

By using Lemma \ref{derivationprop} we can divide the analysis of linear
control systems on $R_{3}, R_{3, \lambda}$ and $R^{\prime}_{3, \lambda}$ as follows:

\textbf{$\bullet$\; \;$G=R_{3}$ or $G=R_{3, \lambda}$ and $\mathcal{D}$ has
only real eigenvalues.}

\begin{proposition}
Let $\Sigma(\mathcal{X},\Delta)$ be a linear control system on $G=R_{3,\lambda
}$ or $R_{3}$. Then,

\begin{itemize}
\item[1.] If $G=R_{3, \lambda}$ the linear system cannot be controllable;

\item[2.] If $G=R_{3}$ the linear system is controllable if and only if it
satisfies the LARC and $\mathfrak{g}=\mathfrak{g}^{0}$ with $\mathcal{D}%
^{\ast}\not \equiv 0$.
\end{itemize}
\end{proposition}

\begin{proof}
Let us analyze the possibilities for $\mathcal{D}^{*}$.

$\bullet \; \dim \ker \mathcal{D}^{\ast}=0.$ In this case $\mathcal{D}^{\ast}$ is
invertible and we can consider the induced system on $F\setminus G$ given by
$\dot{v}=\mathcal{D}^{\ast}v-u\theta(v-\xi)$ which in coordinates reads as
\begin{equation}
\left \{
\begin{array}
[c]{l}%
\dot{x}=\alpha x+by-u\left(  (x-\xi_{1})+\delta(y-\xi_{2})\right) \\
\dot{y}=\beta y-u\lambda(y-\xi_{2})
\end{array}
\right.  ,\label{5}%
\end{equation}
where $\xi=(\xi_{1},\xi_{2})$, $\alpha,\beta \in \mathbb{R}^{\ast}$,
$|\lambda|\in(0,1]$ and $\delta \in \{0,1\}$. Such system is not controllable
since the line $y=\xi_{2}$ works as a barrier for its solutions. In fact, if
for instance $\beta \xi_{2}\geq0$, we have that on points of the form
$(x,\xi_{2})$ it holds that $\dot{y}\geq0$ showing that the solutions starting
on the upper half-plane $\mathcal{C}^{+}=\{(x,y)\in \mathbb{R}^{2};y\geq \xi
_{2}\}$ will remain there (Figure \ref{fig1}). Hence $\Sigma(\mathcal{X}%
,\Delta)$ cannot be controllable.

$\bullet \; \dim \ker \mathcal{D}^{\ast}=1\; \mbox{ and }\;G=R_{3,\lambda}.$ In
this case $\mathcal{D}^{\ast}$ admits two distinct eigenvalues. Moreover, if
$w_{0}\in \ker \mathcal{D}^{\ast}$ is nonzero, the quotient $H_{w_{0}}\setminus
G$ is isomorphic to $\mathrm{Aff}_{0}(\mathbb{R})$ and the induced linear
system admits a nonzero eigenvalue. By Theorem \ref{DathJouan} such system
cannot be controllable and consequently $\Sigma(\mathcal{X},\Delta)$ is not controllable.

$\bullet \; \dim \ker \mathcal{D}^{\ast}=1\; \mbox{ and }\;G=R_{3}.$ Since
$\mathcal{D}^{\ast}$ and $\theta$ commutes $\ker \mathcal{D}^{\ast}$ is
$\theta$-invariant and consequently $\ker \mathcal{D}^{\ast}=\mathrm{span}%
\{e_{1}\}$. By Proposition \ref{derivationprop} the eigenvalue zero is of
multiplicity two for $\mathcal{D}^{\ast}$ implying that $\mathfrak{g}%
=\mathfrak{g}^{0}$. On the other hand, since $\Sigma(\mathcal{X},\Delta)$
satisfies the LARC we must have that $\xi=(\xi_{1},\xi_{2})$ with $\xi_{2}%
\neq0$ and therefore $\{ \xi,\mathcal{D}^{\ast}\xi \}$ is a linear independent
set. Thus, $\Sigma(\mathcal{X},\Delta)$ satisfies the ad-rank condition. By
Theorem \ref{meuteorema} we have the controllability of $\Sigma(\mathcal{X}%
,\Delta)$.

$\bullet \; \dim \ker \mathcal{D}^{\ast}=2.$ Let us assume w.l.o.g. that
$\Sigma(\mathcal{X},\Delta)$ satisfies the LARC. By \ref{projection2}, for any
$w_{0}\in \mathbb{R}^{2}$ the induced system on the homogeneous space
$H_{w_{0}}\setminus G$ is given in coordinates by
\begin{equation}
\left \{
\begin{array}
[c]{l}%
\dot{t}=u\\
\dot{x}=\langle \Lambda_{t}\xi,v_{0}\rangle
\end{array}
\right.  \; \; \mbox{ where }\; \; \langle w_{0},v_{0}\rangle=0\label{6}%
\end{equation}
and we have the following possibilities:

\begin{itemize}
\item[1.] In $R_{3}$ it holds that $\xi=(\xi_{1}, \xi_{2})$ with $\xi_{2}%
\in \mathbb{R}^{*}$. By considering $w_{0}=\theta^{-1}\xi$ the induced system
becomes
\begin{equation}
\left \{
\begin{array}
[c]{l}%
\dot{t}=u\\
\dot{x}=\xi_{2}^{2}t\mathrm{e}^{t}%
\end{array}
\right.
\end{equation}
which is certainly noncontrollable.

%Notar que
%$$\rho_t=\left(\begin{array}{cc}
%\rme^t & t\rme^t \\ 0 & \rme^t
%\end{array}\right)$$

\item[2.] In $R_{3,\lambda}$ it holds that $\xi=(\xi_{1},\xi_{2})$ with
$\xi_{1},\xi_{2}\in \mathbb{R}^{\ast}$. By considering $w_{0}=\lambda(\xi
_{2}^{-1},\xi_{1}^{-1})$ the induced system becomes
\begin{equation}
\left \{
\begin{array}
[c]{l}%
\dot{t}=u\\
\dot{x}=\lambda(\mathrm{e}^{t}-1)-(\mathrm{e}^{\lambda t}-1)
\end{array}
\right.
\end{equation}
which is certainly noncontrollable since $\dot{x}\leq0$ if $\lambda \in
\lbrack-1,0)$ and $\dot{x}\geq0$ if $\lambda \in(0,1]$.
\end{itemize}

In both cases, $\mathcal{D}^{*}\equiv0$ implies that $\Sigma(\mathcal{X},
\Delta)$ cannot be controllable, which concludes the proof.
\end{proof}

\begin{figure}[h]
	\begin{center}
		\includegraphics[scale=.8]{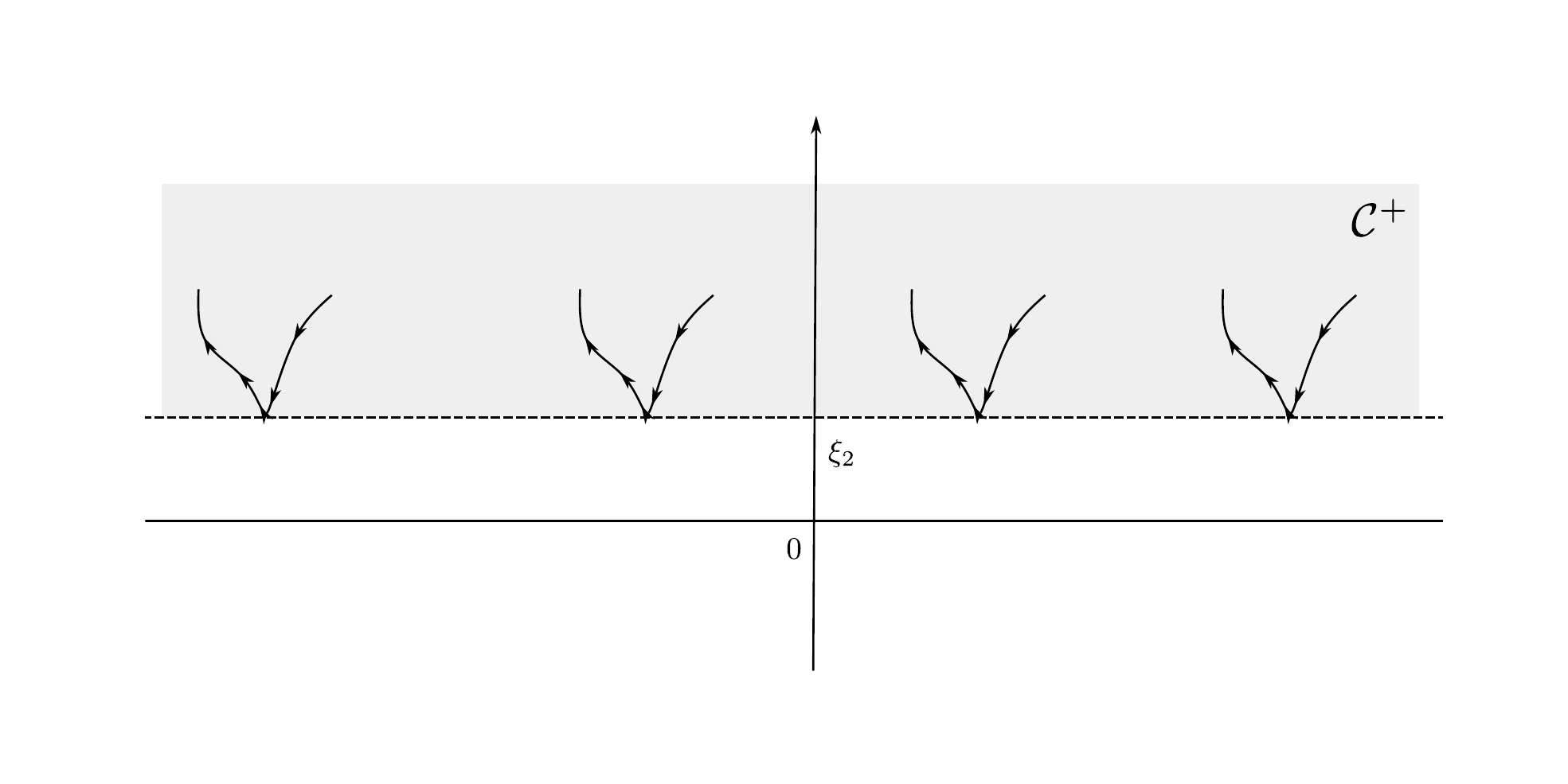}
	\end{center}
	\caption{Invariant region $\mathcal{C}^{+}$ for $\beta>0$.}%
	\label{fig1}%
\end{figure}

\textbf{$\bullet$\; \;$\mathcal{D}$ has a pair of complex eigenvalues and
$G=R_{3, 1}$ or $G=R^{\prime}_{3, \lambda}$}

\begin{proposition}
A linear control system $\Sigma(\mathcal{X}, \Delta)$ on $G$ is controllable
if and only if it satisfies the LARC.
\end{proposition}

\begin{proof}
Let us assume that $\Sigma$ satisfies the LARC and by Proposition
\ref{conjugation} that $(1,0)\in \Delta$. We have three possibilities to consider:

$\bullet$ $\{ \mathcal{D}^{\ast},\theta \}$ is linearly independent. In this
case, it holds that $\mathcal{D}^{\ast}$ is invertible and we can consider the
induced system on the homogeneous space $F\setminus G\simeq \mathbb{R}^{2}$
given by
\begin{equation}
\dot{v}=\mathcal{D}^{\ast}v+u(\theta v-\xi),\; \; \mbox{ where }\; \;(0,\xi
)=\mathcal{D}(1,0)\neq0.\label{3}%
\end{equation}
Since $\{ \mathcal{D}^{\ast},\theta \}$ is linear independent, it holds that the
associate bilinear system $\dot{w}=(\mathcal{D}^{\ast}+u\theta)w$ satisfies

\begin{itemize}
\item[(i)] There exists $u\in \mathbb{R}$ such that $\mathcal{D}^{*}+u\theta$
is skew-symmetric,

\item[(ii)] It satisfies the LARC,
\end{itemize}

and hence it is controllable in $\mathbb{R}^{2}\setminus \{0\}$ (see Theorem
3.3 of \cite{Elliot}). Moreover, the fact that $\dot{v}\neq0$ for any
$u\in \mathcal{U}$, implies by Theorem 2 of \cite{Jurd1} that (\ref{3}) is
controllable and so $G=F\cdot \mathcal{A}=F\cdot \mathcal{A}^{*}$.

If $\mathcal{D}$ admits a pair of complex eigenvalues then $\Sigma$ satisfies
the ad-rank condition and by Lemma \ref{translation} and Theorem
\ref{meuteorema} it holds that $\Sigma(\mathcal{X},\Delta)$ is controllable.
On the other hand, if $\mathcal{D}^{\ast}=\alpha \operatorname{id}$ then
equation \ref{linearflow} gives us that
\[
\varphi_{s}(t(1,0))=\left(  t,(\mathrm{e}^{s\alpha}-1)\alpha^{-1}\Lambda
_{t}\xi \right)  ,t,s\in \mathbb{R}.
\]
If $s>0$ and $\alpha<0$ we obtain
\[
\mathcal{A}\ni \varphi_{s}(t(1,0))=\left(  t,(\mathrm{e}^{s\alpha}%
-1)\alpha^{-1}\Lambda_{t}\xi \right)  \rightarrow(t,-\alpha^{-1}\Lambda_{t}%
\xi)\; \; \; \mbox{ as }\; \; \;s\rightarrow+\infty.
\]
Analogously, if $\alpha>0$ we get
\[
\mathcal{A}^{\ast}\ni \varphi_{-s}(t(1,0))=\left(  t,(\mathrm{e}^{-s\alpha
}-1)\alpha^{-1}\Lambda_{t}\xi \right)  \rightarrow(t,-\alpha^{-1}\Lambda_{t}%
\xi)\; \; \; \mbox{ as }\; \; \;s\rightarrow+\infty.
\]
By Lemma \ref{translation}, in any case $(t,-\alpha^{-1}\Lambda_{t}\xi
)\in \operatorname{cl}(\mathcal{A})\cap \operatorname{cl}(\mathcal{A}^{\ast})$
for any $t\in \mathbb{R}$. Since $F=\{(t,-\alpha^{-1}\Lambda_{t}\xi
),t\in \mathbb{R}\}$ we have that $F\subset \overline{\mathcal{A}}\cap
\overline{\mathcal{A}^{\ast}}$ implying by Lemma \ref{translation} and the
LARC that $\Sigma(\mathcal{X},\Delta)$ is controllable.

$\bullet$ $\mathcal{D}^{\ast}=a\theta$ with $a\neq0$. In this case, the
induced system on $F\setminus G$ has the form
\[
\dot{v}=(a-u)\theta v+u\xi
\]
and we get that
\begin{equation}
\varphi(t,v,u)=\left \{
\begin{array}
[c]{ccc}%
v+ta\xi, & \mbox{ if }u\equiv a & \\
\rho_{at}v & \mbox{ if }u\equiv0 &
\end{array}
\right.  .
\end{equation}
Let us assume $a,a\lambda \in \mathbb{R}^{+}$, since the other possibilities are
analogous. For any given $p,q\in \mathbb{R}^{2}$ we construct a trajectory from
$p$ to $q$ as follows (see Figure \ref{fig3}):

\begin{itemize}
\item[1.] If $p=0$ consider $p^{\prime}=0+t_{0}a\xi$ where $t_{0}$ is any
positive real number. If $p\neq0$ consider $p^{\prime}=p$;

\item[2.] Go through the spiral $\rho_{at}p^{\prime}$ from $p^{\prime}$ to a
point $q^{\prime}$ of the form $q+t_{1}a\xi$ where $t_{1}\leq0$;

\item[3.] Go from $q^{\prime}$ to $q$ through the line $q^{\prime}+ta\xi$.
\end{itemize}

\begin{figure}[th]
\begin{center}
\includegraphics[scale=.7]{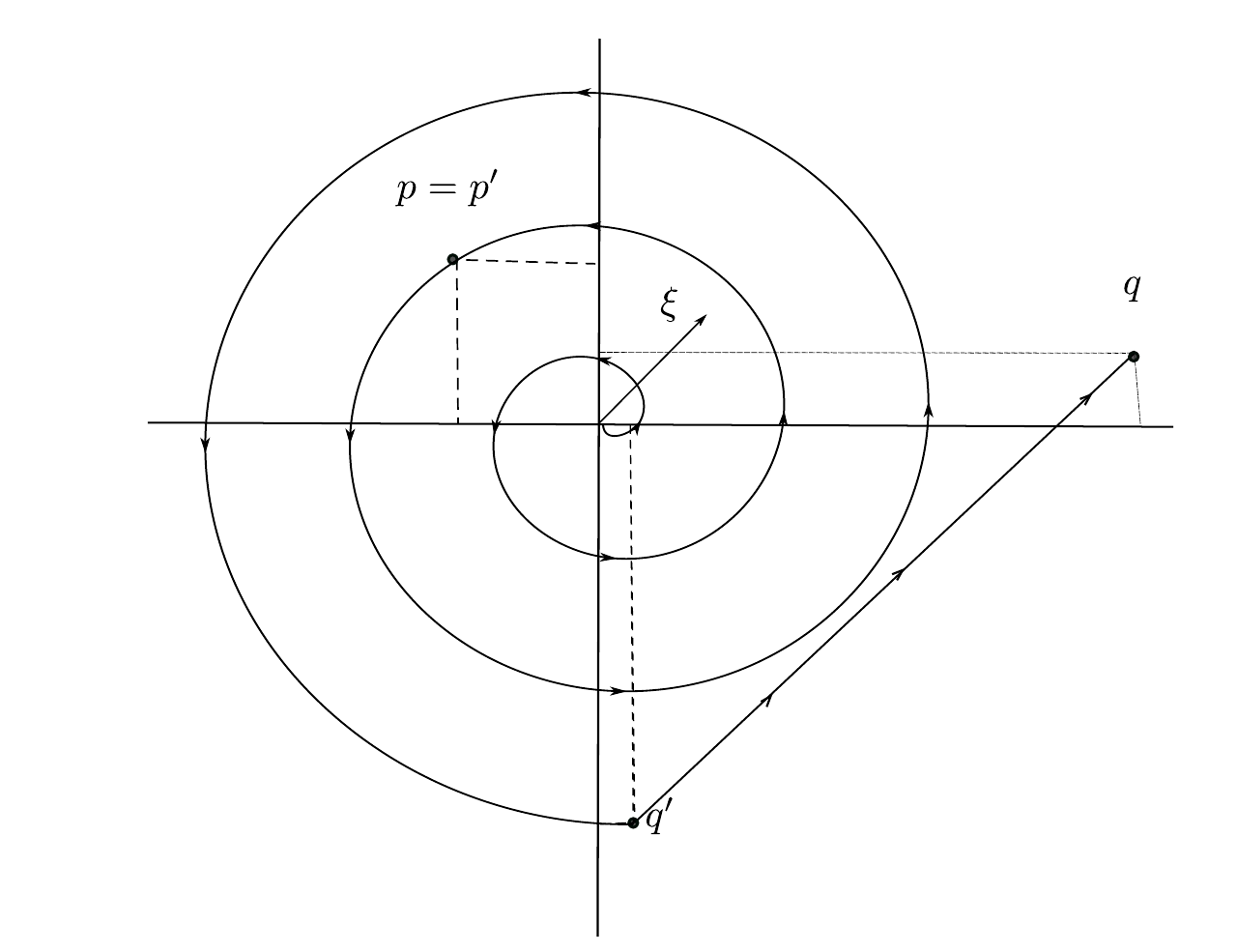}
\end{center}
\caption{Solution connecting $p, q\in \mathbb{R}^{2}$.}%
\label{fig3}%
\end{figure}

$\bullet$ $\mathcal{D}^{\ast}$ is identically zero. By using the fact that
$\mathcal{D}^{\ast}\equiv0$ a simple calculation gives us that $\mathcal{D}%
=\operatorname{ad}(0,\zeta)$ where $\zeta=-\theta^{-1}\xi$. Moreover, since
$H\subset \mathcal{A}_{s}$ for any $s>0$ we have
\begin{equation}
(t,0)\varphi_{s}(-t,0)=(t,0)(-t,(1-\rho_{-t})s\zeta)=(0,(\rho_{t}%
-1)s\zeta)=(0,s\Lambda_{t}\zeta)\in \mathcal{A},\; \; \; \mbox{ for all }\; \;t\in
\mathbb{R},\;s>0.\label{final}%
\end{equation}
However, the fact that $\{ \Lambda_{t}\zeta,t\in \mathbb{R}\}$ is a spiral
implies that $\{s\Lambda_{t}\zeta,\;t\in \mathbb{R},s>0\}=\mathbb{R}^{2}$ and
consequently that $\mathbb{R}^{2}\subset \mathcal{A}$. By Lemma
\ref{translation} we have also that $\mathbb{R}^{2}\subset \mathcal{A}^{\ast}$
and therefore,
\[
G=\mathbb{R}^{2}\cdot H\subset \mathbb{R}^{2}\cdot \mathcal{A}\subset
\mathcal{A}\; \; \; \mbox{ and }\; \; \;G=\mathbb{R}^{2}\cdot H\subset
\mathbb{R}^{2}\cdot \mathcal{A}^{\ast}\subset \mathcal{A}^{\ast}%
\]
concluding the proof.
\end{proof}

\subsection{The two-dimensional case}

For the two-dimensional case it holds that

\begin{theorem}
Let $\Sigma(\mathcal{X},\Delta)$ be a linear control system on a
three-dimensional nonnilpotent solvable Lie group $G$ that satisfies the LARC
and $\dim \Delta=2$. It holds:

\begin{itemize}
\item[1.] If $G=R_{2}$ or $\widetilde{R_{2}}$: $\Sigma(\mathcal{X}, \Delta)$
is controllable if and only if $\dim \mathfrak{g}^{0}>1$ or $\dim
\mathfrak{g}^{0}=1$ and $\Delta \simeq \mathfrak{aff}(\mathbb{R})$;

\item[2.] If $G=E_{n}, \widetilde{E}$ or $R^{\prime}_{3, \lambda}$:
$\Sigma(\mathcal{X}, \Delta)$ is controllable;

\item[3.] If $G=R_{3}$: $\Sigma(\mathcal{X}, \Delta)$ is controllable if and
only if $\mathfrak{g}=\mathfrak{g}^{0}$;

\item[4.] If $G=R_{3, \lambda}$: $\Sigma(\mathcal{X}, \Delta)$ is controllable
if and only if $\ker \mathcal{D}^{*}\not \subset \Delta$ or $\mathcal{D}$ has a
pair of complex eigenvalues.
\end{itemize}
\end{theorem}

\subsubsection{The case $G=R_{2}$ or $G=\widetilde{R_{2}}$.}

\begin{theorem}
Let $\Sigma(\mathcal{X}, \Delta)$ be a linear control system on $G$. Then
$\Sigma(\mathcal{X}, \Delta)$ is controllable if and only if it satisfies the
LARC and $\dim \mathfrak{g}^{0}>1$ or $\dim \mathfrak{g}^{0}=1$ and
$\Delta \simeq \mathfrak{aff}(\mathbb{R})$
\end{theorem}

\begin{proof}
Let us assume that $\Sigma(\mathcal{X},\Delta)$ satisfies the LARC. By
Propositions \ref{simplesmente} and \ref{conjugation} we only have to consider
a linear control system $\Sigma(\mathcal{X},\Delta)$ on $\widetilde{R_{2}}$
such that $\Delta=\mathrm{span}\{(1,0),(0,w)\}$. Moreover,
\[
\Delta \ni \lbrack(1,0),(0,w)]=(0,\theta w)\implies \theta w\in \mathrm{span}\{w\}
\]
and therefore $\Delta=\mathrm{span}\{(1,0),(0,e_{2})\} \simeq \mathrm{Aff}%
_{0}(\mathbb{R})$ or $\Delta=\mathrm{span}\{(1,0),(0,e_{1})\}$ is Abelian.

\begin{itemize}
\item[1.] If $\dim \mathfrak{g}^{0}=3$ then $\mathfrak{g}=\mathfrak{g}^{0}$ and
since in this case the LARC is equivalent to the ad-rank condition Theorem
\ref{meuteorema} implies the controllability of $\Sigma(\mathcal{X}, \Delta)$;

\item[2.] If $\dim \mathfrak{g}_{0}=2$ we have that $\mathfrak{g}%
=\mathfrak{g}^{+}\Delta$ and consequently $G=G^{0}\cdot G_{\Delta}$. Since
$\Sigma(\mathcal{X}, \Delta)$ satisfies the ad-rank condition Theorem
\ref{meuteorema} implies $G^{0}\subset \mathcal{A}\cap \mathcal{A}^{*}$ and by
Lemma \ref{translation}
\[
G=G^{0}\cdot G_{\Delta}\subset G^{0}\cdot \mathcal{A}\subset \mathcal{A}%
\; \; \mbox{ and }\; \;G=G^{0}\cdot G_{\Delta}\subset G^{0}\cdot \mathcal{A}%
^{*}\subset \mathcal{A}^{*}%
\]
implying the controllability of $\Sigma(\mathcal{X}, \Delta)$.

\item[3.] If $\dim \mathfrak{g}^{0}=1$ then necessarily $\mathcal{D}^{\ast}$ is invertible.

\begin{itemize}
\item If $\Delta=\mathrm{span}\{(1,0),(0,e_{2})\}$ we can consider the induced
linear system on $H_{e_{2}}\setminus G\simeq \mathbb{R}^{2}$. Since such system
satisfy the ad-rank condition it is controllable implying that $G=H_{e_{2}%
}\cdot \mathcal{A}=H_{e_{2}}\cdot \mathcal{A}^{\ast}$. On the other hand, the
fact that $H_{e_{2}}$ is $\varphi$-invariant and $H_{e_{2}}\subset G_{\Delta
}\subset \mathcal{A}\cap \mathcal{A}^{\ast}$ implies by Lemma \ref{translation}
the controllability of $\Sigma(\mathcal{X},\Delta)$.

\item If $\Delta=\mathrm{span}\{(1,0),(0,e_{1})\}$, the induced system on the
two-dimensional solvable Lie group \linebreak$H_{e_{1}}\setminus
G\simeq \mathrm{Aff}_{0}(\mathbb{R})$ cannot be controllable. In fact, the
induced derivation admits a nonzero eigenvalue and consequently $\Sigma
(\mathcal{X},\Delta)$ cannot be controllable, concluding the proof.
\end{itemize}
\end{itemize}
\end{proof}

\subsubsection{The case $G=R_{3}, \;R_{3, \lambda}$, $R^{\prime}_{3, \lambda}%
$, $E_{n}$ or $\widetilde{E_{2}}$.}

Let us separate the cases as follows:

\textbf{$\bullet$\; \;$G=R_{3}$ or $G=R_{3, \lambda}$.}

\begin{proposition}
Let $\Sigma(\mathcal{X},\Delta)$ be a linear control system on $G=R_{3,\lambda
}$ or $R_{3}$. Then,

\begin{itemize}
\item[1.] If $G=R_{3, \lambda}$ the linear system is controllable if and only
if it satisfies the LARC and $\ker \mathcal{D}^{*}\not \subset \Delta$ or
$\mathcal{D}^{*}$ has a pair of complex eigenvalues;

\item[2.] If $G=R_{3}$ the linear control system is controllable if and only
if it satisfies the LARC and $\mathfrak{g}=\mathfrak{g}^{0}$.
\end{itemize}
\end{proposition}

\begin{proof}
We can as before assume that $\Delta=\mathrm{span}\{(1, 0), (0, w)\}$ where
$w\in \mathbb{R}^{2}$ is a common eigenvector of $\theta$ and $\mathcal{D}^{*}$.

1. If $\mathcal{D}^{*}$ has a pair of complex eigenvalues $G=R_{3, 1}$ and the
linear control system $\Sigma(\mathcal{X}, \Delta)$ is controllable if and
only if it satisfies the LARC by the one-dimensional case. Let us then assume
that $\mathcal{D}^{*}$ has a pair of real eigenvalues and analyze the
dimension of $\ker \mathcal{D}^{*}$.

$\bullet \; \dim \ker \mathcal{D}^{\ast}=0.$ In this case $\ker \mathcal{D}^{\ast
}=\{0\} \subset \Delta$. By Theorem \ref{DathJouan} the induced linear control
system on $H_{w}\setminus G\simeq \mathrm{Aff}_{0}(\mathbb{R}) $ cannot be
controllable, since the associated derivation has a nonzero eigenvalue.
Therefore, $\Sigma(\mathcal{X},\Delta)$ can not be controllable.

$\bullet \; \dim \ker \mathcal{D}^{\ast}=1.$ If $\ker \mathcal{D}^{\ast
}=\mathrm{span}\{w\} \subset \Delta$ we have as before, that the derivation of
the induced linear control system on $H_{w}\setminus G\simeq \mathrm{Aff}%
_{0}(\mathbb{R})$ has a nonzero eigenvalue and therefore cannot be
controllable implying that $\Sigma(\mathcal{X},\Delta)$ is not controllable.

On the other hand, if $\ker \mathcal{D}^{\ast}\not \subset \Delta$ it follows
that $\mathbb{R}^{2}=\ker \mathcal{D}^{\ast}\oplus \mathbb{R}w\subset
\mathcal{A}\cap \mathcal{A}^{\ast}$ and consequently
\[
G=\mathbb{R}^{2}\cdot H\subset \mathbb{R}^{2}\cdot \mathcal{A}\subset
\mathcal{A}\; \; \mbox{ and }\; \;G=\mathbb{R}^{2}\cdot H\subset \mathbb{R}%
^{2}\cdot \mathcal{A}^{\ast}\subset \mathcal{A}^{\ast}%
\]
implying the controllability of $\Sigma(\mathcal{X},\Delta)$.

$\bullet \; \dim \ker \mathcal{D}^{\ast}=2.$ In this case we have that
$\mathfrak{g}=\mathfrak{g}^{0}$ which by the Theorem \ref{meuteorema}, and the
fact that $\Sigma(\mathcal{X},\Delta)$ satisfies the ad-rank condition implies
the controllability.

2. Since $G=R_{3}$ we necessarily have $w=e_{1}$. Moreover, if $\dim
\ker \mathcal{D}^{\ast}\geq1$ then $\mathfrak{g}=\mathfrak{g}^{0}$ which by
Theorem \ref{meuteorema} implies the controllability of $\Sigma(\mathcal{X}%
,\Delta)$. Therefore, we only have to show that when $\mathcal{D}^{\ast}$ is
invertible, $\Sigma(\mathcal{X},\Delta)$ cannot be controllable. In order to
do that let us consider the induced system on $F\setminus G$ given by $\dot
{v}=\mathcal{D}^{\ast}v-u_{1}\theta(v-\xi)+u_{2}\mathcal{D}^{\ast}e_{1}$. In
coordinates we have that
\begin{equation}
\left \{
\begin{array}
[c]{l}%
\dot{x}=\alpha x+by-u_{1}\left(  (x-\xi_{1})+(y-\xi_{2})\right)  +\alpha
u_{2}\\
\dot{y}=\beta y-u(y-\xi_{2})
\end{array}
\right.  ,\label{9}%
\end{equation}
where $\xi=(\xi_{1},\xi_{2})$ and $\alpha,\beta \in \mathbb{R}^{\ast}$. As for
the one-dimensional case, such system is not controllable since the line
$y=\xi_{2}$ works as a barrier for its solutions (see Figure \ref{fig2}),
concluding the proof.
\end{proof}

\textbf{$\bullet$\; \;$G=E_{n}, \widetilde{E}$ or $G=R^{\prime}_{3, \lambda}$.}

When $G=E_{n}, \widetilde{E}$ or $G=R^{\prime}_{3, \lambda}$ the only
two-dimensional subalgebra of their associated Lie algebra is $\mathbb{R}^{2}$
as follows from Lemma \ref{subalgebra}. Therefore, any linear system
$\Sigma(\mathcal{X}, \Delta)$ with $\dim \Delta>1$ is trivially controllable if
it satisfies the LARC, since in this case we necessarily have that $\dim
\Delta=3 $.

\section*{Acknowledgements}

The first author was supported by Proyectos Fondecyt n$^{o}$ 1150292 and n$%
%TCIMACRO{\U{b0}}%
%BeginExpansion
{{}^\circ}%
%EndExpansion
$ 1150292, Conicyt, Chile and the second one was supported by Fapesp grants
n$^{o}$ 2016/11135-2 and n$^{o}$ 2018/10696-6.

%We would like to thank the Centro de Estudios Cient\'{\i}ficos (CECs) in
%Valdivia, Chile, through Prof. Dr. Jorge Zanelli, for providing to the first and
%second authors an excellent environment to work out on this article.

\end{document}